\documentclass[10pt]{article}   
\usepackage{amsfonts,amsmath,latexsym}
\usepackage[T1]{fontenc}
\usepackage{dsfont}
\usepackage{hyperref}
\usepackage[dvips]{graphicx}
\usepackage{epsfig}
\usepackage{enumerate}
\usepackage{epsf}   
\usepackage{amsthm}
\usepackage{color}
\usepackage{amssymb}

\usepackage{caption}
\usepackage{subfigure}

\usepackage{fancyhdr} 

\usepackage{palatino}

\newtheorem{thm}{Theorem}[section]
\newtheorem{prop}[thm]{Proposition}
\newtheorem{lem}[thm]{Lemma}
\newtheorem{corro}[thm]{Corollary}
\newtheorem{defi}[thm]{Definition}
\newtheorem{rem}[thm]{Remark}

\newtheorem{ass}{Assumption}

\def\R{\mathbb R}
\def\N{\mathbb N}

\def\E{\mathbb E}
\def\P{\mathbb P}

\def\shc{{\cal C}}

\def\shf{{\cal F}}

\def\shm{{\cal M}}

\def\shp{{\cal P}}
\def\shs{{\cal S}}

\author{
{\sc Anthony LE CAVIL}
\thanks{ENSTA-ParisTech, Universit\'e Paris-Saclay.
 Unit\'e de Math\'ematiques Appliqu\'ees (UMA).
  E-mail:{ \tt anthony.lecavil@ensta-paristech.fr}} 
 {\sc,}\ {\sc Nadia OUDJANE}
\thanks{EDF R\&D,   and FiME (Laboratoire de Finance des March\'es de l'Energie
(Dauphine, CREST,  EDF R\&D) www.fime-lab.org). 
E-mail:{\tt  
nadia.oudjane@edf.fr}}
\ {\sc and}\ {\sc Francesco RUSSO} 
\thanks{ENSTA-ParisTech, Universit\'e Paris-Saclay.
 Unit\'e de Math\'ematiques Appliqu\'ees (UMA). 
  E-mail:{\tt  francesco.russo@ensta-paristech.fr}
The financial support of this author was partially provided
 by the DFG 
through the CRC ''Taming uncertainty and profiting from randomness and
low regularity in analysis, stochastics and their application''.
}.
}

\date{September 11th 2017}
\title{Monte-Carlo Algorithms for Forward Feynman-Kac type representation for semilinear
nonconservative Partial Differential Equations}

\newcommand{\MBFigure}[6]{
$\left. \right.$ \\
\refstepcounter{figure}
\addcontentsline{lof}{figure}{\numberline{\thefigure}{\ignorespaces #5}}
\begin{center}
\begin{minipage}{#1cm}
\centerline{\includegraphics[width=#2cm,angle=#3]{#4}}
\begin{center}
\upshape{F\textsc{ig} \normal
\end{center}
size{\thefigure}. $-$} #5
\end{center}
\label{#6}
\end{minipage}
\end{center}
$\left. \right.$ \\}

\begin{document}
\maketitle 
 \begin{abstract}
 The paper  is devoted to the construction of a probabilistic particle 
algorithm.  
 This is related to  nonlinear forward  Feynman-Kac type equation,
 which represents the solution of
a nonconservative semilinear parabolic Partial Differential Equations (PDE).
Illustrations of the efficiency of 
the algorithm are provided by  numerical experiments. 
 \end{abstract}
 \medskip\noindent {\bf Key words and phrases:}  
Semilinear Partial Differential Equations; Nonlinear Feynman-Kac type functional; Particle systems; Euler schemes.
 
 \medskip\noindent  {\bf 2010  AMS-classification}: 60H10; 60H30; 60J60; 65C05; 65C35; 68U20; 35K58.
 
 % % % % % % % % % % % % % % % % % % % % % % % % % % % % % % %
 
\section{Introduction}

In this paper, we  consider a forward probabilistic representation of the 
 semilinear Partial Differential Equation (PDE) on $[0,T]\times \R^d$ 
\begin{equation}
\label{eq:PDE}
\left \{ 
\begin{array}{l}
\partial_t u = L_t^{\ast} u + u \Lambda(t,x,u,\nabla u) \\
u(0,\cdot) = u_0 \ ,
\end{array}
\right .
\end{equation}
where $u_0$ is a Borel probability measure on $\R^d$ and  $L^{\ast}$ is a partial differential operator of the type
\begin{equation} \label{eq:AdjGen2}
(L^\ast_t \varphi)(x) = \frac{1}{2} \sum_{i,j=1}^d \partial_{ij}^2  (a_{i,j}(t,x)
 \varphi)(x) - \sum_{i=1}^d  \partial_{i} (g_i(t,x)  \varphi)(x), \quad \textrm{ for } \varphi \in \shc_0^{\infty}(\R^d).
\end{equation}
% $L^{\ast}$ is formally the operator given in
% $$
% (L_t \varphi)(x) = \frac{1}{2} \sum_{i,j=1}^d a_{i,j}(t,x) \partial_{ij}^2 \varphi(x) + \sum_{i=1}^d g_i(t,x) \partial_{i} \varphi(x), \quad \textrm{ for } \varphi \in \shc_0^{\infty}(\R^d),
% $$
% where $a = (a_{i,j}) $ and $g$ are bounded, measurable functions on
% $[0,T] \times \R^d$ with values in the space of $d \times d $ matrices
%$\shs_d$ and  $\R^d$ respectively. \\

%When $\Lambda = 0$, equation \eqref{eq:PDE} is reduced to the classical Fokker-Planck PDE, that has been extensively studied and for which very general existence/uniqueness results have been proved, see e.g. \cite{bogachevkrylovBook, BCR2}.
% (citer Bogachev,Rockner).
In this specific case, a forward probabilistic representation of \eqref{eq:PDE} is related to the solution $Y$ of the Stochastic Differential Equation (SDE) associated with the infinitesimal generator $L$ and the initial condition  $u_0$,   i.e.
\begin{equation}
\label{eq:SDE1}
\left \{ 
\begin{array}{l}
Y_t = Y_0 + \int_0^t \Phi(s,Y_s) dW_s + \int_0^t g(s,Y_s) ds \\
Y_0 \sim u_0 \ ,
\end{array}
\right .
\end{equation}
with $\Phi \Phi^t = a$. % Existence  and uniqueness of $Y$ will  be understood in
%the weak sense, i.e. in law.
More precisely, if \eqref{eq:SDE1} admits a solution $Y$, then 
 the marginal laws $(u_t(dx), t \geq 0)$ of $(Y_t,t \geq 0)$ satisfy the Fokker-Planck (also called
forward Kolmogorov) equation, which corresponds to PDE
\eqref{eq:PDE} when $\Lambda = 0$. In this sense, the couple $(Y,u)$ is a (forward) 
probabilistic representation of \eqref{eq:PDE}.
\\
In the case where $\Lambda \neq 0$, we propose a representation which is constituted by a
 couple $(Y,u)$, solution of the system
%an adapted variant of the generalized McKean type SDE introduced in \cite{LOR1}, and given here by
\begin{equation}
\label{eq:NLSDE}
\left \{ 
\begin{array}{l}
Y_t = Y_0 + \int_0^t \Phi(s,Y_s) dW_s + \int_0^t g(s,Y_s) ds, \quad Y_0 \sim u_0 \\
\int_{\R^d} \varphi(x) u(t,x)dx = \E\Big[ \varphi(Y_t) \exp\Big(\int_0^t \Lambda(s,Y_s,u(s,Y_s), \nabla u(s,Y_s)) \Big) \Big], \quad \textrm{for } t\in (0,T]\,,\  \varphi \in \shc_b(\R^d) \ .
\end{array}
\right .
\end{equation}
The main starting point of the paper is the following. 
{\it  If $(Y,u)$ is a solution of \eqref{eq:NLSDE}, then $u$  solves \eqref{eq:PDE} 
in the sense of distributions.} 
This follows by a direct application of It\^o formula and integration by parts. \\
%Because of Proposition \ref{FKDistr};
 A function $u$ solving the second line of
\eqref{eq:NLSDE} will be often identified as
%of \eqref{eq:PDE},  will be called 
% 
{\bf{\itshape{Feynman-Kac type representation}}} of \eqref{eq:PDE}. 
% The terminology {\bf Feynman-Kac} is used here because of its analogy with the usual Feynman-Kac representation of a linear Kolmogorov type PDE in its backward form, which is however a  different concept. A generalization of the latter has led to the theory of Markovian forward-backward SDEs of \cite{pardoux}.
%REFLECHIR AU LIEN AVEC L'INTEGRALE DE FEYNMAN \\
We emphasize  that a solution to
%the probabilistic representation 
equation \eqref{eq:NLSDE}  introduced here, is a couple $(Y,u)$, 
 where $Y$ is a process  solving a {\bf{\itshape{classical SDE}}}, and $u : [0,T] \times \R^d \rightarrow \R$ 
satisfies the second line equation of \eqref{eq:NLSDE}. 
 
Equation~\eqref{eq:NLSDE} constitutes a particular  case of 
%One could theoretically consider a variant of \eqref{eq:NLSDE} where the 
%functions $\Phi, g$ depend on
 {\bf{\itshape{McKean  type SDE}}}, where the coefficients $\Phi$ and $g$ do not depend on $u$. 
%Nadia
%Indeed, a solution $(Y,u)$ of that McKean system would be a solution of a more general non-linear equation than \eqref{eq:PDE}.
In \cite{LOR1} and \cite{LOR2} we have fully analyzed a regularized version of 
the McKean type SDE,
where $\Phi, g$ together with $\Lambda$ also depend on the unknown function $u$, but no dependence on $\nabla u$ was considered at that level.
 The first paper focuses on
various results on existence and uniqueness and the second one on numerical approximation schemes.
 Even though, the present paper does not consider any  McKean type non linearity in the SDE, it extends the class of nonlinearities considered in~\cite{LOR1,LOR2} with respect to (w.r.t.) $\nabla u$.
%Concerning the diffusion process $Y$, it is clear that the coefficients $\Phi$, $g$ do not depend on the function $u$. This means there is not McKean dependence, (in the sense of/) contrary to \cite{LOR1}. 
%But we observe here that the nonlinear term in the PDE \eqref{eq:PDEIntro}, i.e.  
Indeed, in the present paper, the dependence of $\Lambda$ appears to be more singular than in \cite{LOR1,LOR2}, since it involves not only $u$ but also $\nabla u$ allowing to cover a different class of semilinear PDEs of the form~\eqref{eq:PDE}. 
%the latter one does not  involve  a nonlinear dependence w.r.t. $\nabla u$ and moreover the dependence on $u$ is regularized.  
% Besides, the probabilistic representation~\eqref{eq:NLSDE} will be directly related to the target PDE~\eqref{eq:PDE}, whereas \cite{LOR1} provides a probabilistic representation  related to an integro-differential equation corresponding to a regularization of the target PDE, without any theoretical result ensuring the convergence of the regularization to the target PDE.
%
The companion paper~\cite{LOR3} focuses on the theoretical aspects of \eqref{eq:PDE}.
In this article we propose an associated numerical approximation scheme. 
% numerical scheme have also been the object of extensive developments in the literature. One major approach which has been largely investigated for
%approximating solutions of time evolutionary PDEs is the
%forward-backward SDEs (FBSDEs) method.  FBSDEs were initially developed
% in~\cite{pardoux}, see
%also \cite{pardouxgeilo} for a survey and \cite{rascanu} for a recent monograph on the subject.  
% The idea is to express the PDE solution $v(t,\cdot)$ at time $t$  as the expectation of
% a functional of a so called forward diffusion process $X$, starting at time $t$.
%  Based on that idea, many judicious numerical schemes have been proposed by~\cite{BouchardTouzi,GobetWarin}. 
%However,  

An important part of the literature for approaching semilinear 
PDEs is based on  Forward Backward Stochastic Differential Equations (FBSDEs) initially developed
 in~\cite{pardoux}, see
also \cite{pardouxgeilo} for a survey and \cite{rascanu} for a recent monograph on the subject.  
Based on that idea, many judicious numerical schemes have been proposed (see for instance~\cite{BouchardTouzi,GobetWarin}). 
All those  rely on computing recursively conditional expectation functions which is known to be a difficult
 task in high  dimension. Besides, the FBSDE approach is \textit{blind} in the sense that the forward process $X$ is
 not ensured to explore the most relevant regions of the space to approximate efficiently the solution of the PDE.
The FBSDE representation of fully nonlinear PDEs still requires complex developments and is 
the subject of active research, see for instance~\cite{cheridito}. 
Branching diffusion processes 
%or in particular superprocesses (see e.g. ~\cite{Dynkin})
 provide alternative  probabilistic representation of semilinear PDEs,
 involving a specific form of non-linearity on the zero order term,
see e.g.  in~\cite{labordere,LabordereTouziTan}.
% to a more general class of non-linearities on the zero order term, with the so-called \textit{marked branching process}. 
%One of the main advantage of this approach compared to FBSDEs is that it does not involve any regression computation
%to calculate conditional expectations.
 More recently, an extension of the branching diffusion representation to a class of semilinear PDEs has been proposed in~\cite{HOTTW}. 
%A third class of numerical approximation schemes relies on McKean type representations. 
%In this context, Feynman-Kac formula and various types of related particle approximation schemes were extensively 
%analyzed in the reference books of Del Moral~\cite{DelMoral} and~\cite{moral1} but without considering the  specific case
%of a  time continuous system~\eqref{eq:NSDE} coupled with  a weighting function $\Lambda$ 
%which depends  nonlinearly on $u$ and $\nabla u$. 
As mentioned earlier, the main idea of the present paper is to investigate the forward Feynman-Kac type representation~\eqref{eq:NLSDE} allowing to tackle a large class of first order nonlinearities thanks to the dependence of the weighting function $\Lambda$ on  both $u$ and $\nabla u$.
 In the time continuous framework, classical (forward) McKean representations are restricted to the conservative case ($\Lambda=0$).
 At the algorithmic level, \cite{bossytalay1} has  contributed to develop stochastic particle methods  in the spirit of McKean
 to approach a PDE related to Burgers equation
 providing first the rate of convergence.
 Comparison with classical
 numerical analysis techniques was provided by \cite{piperno}.
% Relevant contributions at the algorithmic level are \cite{BossyTalay95,
% BossyTalay97, 
%bossyjourdain, mbajourdain}.
%and the survey paper \cite{Talay}.
In the case $\Lambda = 0$  with $g=0$, but with $\Phi$ possibly discontinuous,
some empirical implementations were conducted in \cite{BCR1, BCR3} in the one-dimensional and multi-dimensional case
respectively,  in order to predict the large time qualitative behavior 
of the solution of the corresponding PDE.
An interesting aspect  of this approach is that it could potentially be extended to represent a specific class of second order nonlinear PDEs, by extending it to the case where $\Phi$ and $g$ also depend on $u$. 
%by considering a more general class of functions $\Phi$, $g$ and $\Lambda$ which may depend non-linearly not only on $u$, $\nabla u$ but also on $\nabla^2 u$. 
This more general setting, extending~\cite{LOR1,LOR2}, will
 be investigated in a future work.\\
%As we have mentioned, the focus of the paper is on \eqref{eq:PDE} and on its probabilistic counterpart \eqref{eq:NLSDE}.
% which constitutes
%its natural probabilistic counterpart.
%Theoretical aspects of this were developed in \cite{LOR3}.  
% The natural interpretation for \eqref{eq:PDE} is the sense of 
% distributions (weak solution), see Definition \ref{def:SolPDE} 1. which is equivalent to the mild sense which is formulated
% with the help of semigroups, see Definition \ref{def:SolPDE} 2. provided there is a unique solution
% when $\Lambda = 0$, see Lemma \ref{lem:MildWeak}. 
%  Under Lipschitz conditions on the second variable on $\Phi, g$,  Theorem \ref{thm:FKForm}
% makes the bridge between analysis and probability showing that \eqref{eq:NLSDE} admits a solution if and only if \eqref{eq:PDE}
% admits a solution. 
% Under Assumption \ref{ass:mainP3}   Proposition \ref{prop:UniTrFun}  shows existence and uniqueness
% of the mild solution in $L^1([0,T], W^{1,1}(\R^d))$. When $\Lambda$ does not depend on the gradient Theorem \ref{CasLambda} 
% shows well-posedness in the larger space  $L^1([0,T], L^{1}(\R^d))$. 
%Summarizing, the first contribution of the paper is the well-posedness of the non-linear diffusion equation
%\eqref{eq:NLSDE}, which is naturally associated with \eqref{eq:PDE}. \\
The  main contribution of this paper is
to propose and analyze an original Monte Carlo scheme \eqref{eq:tildeYuP3} to approximate
the solution of \eqref{eq:NLSDE} and consequently also the solution $u$ of \eqref{eq:PDE} which constitutes an equivalent 
(deterministic)  form.
  This numerical scheme relies on three approximation steps:  a regularization procedure based on a kernel convolution, a space 
discretization based on Monte Carlo simulations of the diffusion $Y$~\eqref{eq:NLSDE} and a time discretization.
%Each resulting error is analyzed properly and their combination allows to establish the convergence of the numerical scheme in Theorem \ref{thm:Cvg}.
%Section \ref{SConverg} analyzes the convergence and its approximation rate in $L^1$ of the solution $u$ of \eqref{eq:PDE}
%via a sequence of approximation $u^{\varepsilon}$ which are associated to the solution of a smoothed form
%of \eqref{eq:PDE} i.e. \eqref{eq:PDEReg}. 
%This is the object of Theorem \ref{thm:cvguu'}  and Corollary \ref{C47}.
In Section \ref{SPartAlgo}, we present our original particle approximation scheme whose convergence is established in
Theorem \ref{thm:Cvg}.
% Proposition \ref{prop:CvgParticle} and 
%Proposition \ref{prop:CvgTime}.
Section \ref{sec:simulations} is finally devoted to numerical simulations.

\section{Preliminaries} %Notations and assumptions
\label{S2P3}

\subsection{Notations}
\label{S21}

\setcounter{equation}{0}

%%%%%%%%%%%%%%%%%%%%%%%%%%%%%%%%%%%
Let $d \in \N^{\star}$. Let us consider $\shc^d:=\mathcal{C}([0,T],\R^d)$ metricized by the supremum norm $\Vert \cdot \Vert_{\infty}$, equipped with its Borel $\sigma-$ field $\mathcal{B}(\shc^d)$
% = \sigma(X_t,t \geq 0)$ (and $\mathcal{B}_t(\shc^d) := \sigma(X_u,0 \leq u \leq t)$ the canonical filtration)
and endowed with the topology of uniform convergence. \\ 
%$X$ will be the canonical process on $\shc^d$.\\
If $(E,d_E)$ is a Polish space,
%Given $r \ge 0$, $\mathcal{P}_r(E)$ is the set of Borel probability measures on $E$ admitting a moment of order $r$. For $r=0$,
 $\mathcal{P}(E)$
% := \mathcal{P}_0(E)$
 denotes the Polish space (with respect to the weak convergence topology) of Borel probability measures on $E$ naturally equipped with its Borel $\sigma$-field $\mathcal{B}(\shp(E))$. 
 The reader can consult
 Proposition 7.20 and Proposition 7.23, 
Section 7.4 Chapter 7 in \cite{BertShre} for more exhaustive information.
When $d=1$,  we simply note 
 $\mathcal{C} :=  \mathcal{C}^1$. $\shc_b(E)$ denotes the space of bounded, continuous real-valued functions on $E$.
%still equipped with the  norm $\Vert \cdot \Vert_{\infty}$. \\
%Given $N \in \N^{\star}$, $l \in \shc^d$, $l^1, \cdots, l^N \in \shc^d$, a significant role in this paper will be played by the Borel measures on $\shc ^d$ given by $\delta_l$ and $\displaystyle{ \frac{1}{N} \sum_{j=1}^N \delta_{l^j}}$.\\
 
In this paper, $\R^d$ is equipped with the Euclidean scalar product $\cdot $ and $\vert x\vert $  stands for the induced norm for $x \in \R^d$. 
The gradient operator 
% (w.r.t. $x \in \R^d$) for differentiable 
for functions defined on $\R^d$ is denoted by  $\nabla$.
If a function $u$ depends on a variable $x \in \R^d$ and other variables,
we still denote by $\nabla u$ the gradient of $u$ with respect
to $x$, if there is no ambiguity.
% $\nabla$ will simply denote the gradient on $\R^d$.
%Given two reals $a$ and $b$, we set $a \wedge b := \min(a,b)$ and $a \vee b := \max(a,b)$.
$M_{d,p}(\R)$ denotes the space of $\R^{d \times p}$ real matrices equipped with the Frobenius norm (also denoted $\vert \cdot \vert$), i.e. the one induced by the scalar product $(A,B) \in M_{d,p}(\R) \times M_{d,p}(\R)  \mapsto Tr(A^tB)$,
 where $A^t$ stands for the transpose matrix of $A$ and $Tr$ is the trace operator. $\shs_d$ is the set of symmetric, non-negative definite $d \times d$ real matrices and $\shs_d^+$ the set of strictly positive definite matrices of $\shs_d$. \\
$\mathcal{M}_f(\R^d)$
%(resp. $\shm_f^+(\R^d)$)
is the space of finite
%(resp. finite, positive)
Borel measures on $\R^d$. 
%When it is endowed with the weak convergence topology, $\shb(\shm_f(\R^d))$
%(resp. $\shb(\shm_{f}^+(\R^d))$)
%stands for its Borel $\sigma$-field. 
$\Vert \cdot \Vert_{TV} $ denotes the associated total variation distance. \\
%It is well-known that $(\shm_f(\R^d),\Vert \cdot \Vert_{TV})$ is a Banach space, where $\Vert \cdot \Vert_{TV}$ denotes the total variation norm. \\
%$\mathcal{S}(\R^d)$ is the space of Schwartz fast decreasing test functions and $\mathcal{S}'(\R^d)$ is its dual. 
$\mathcal{C}_b(\R^d)$ is the space of bounded, continuous functions on  $\R^d$ and $\mathcal{C}^{\infty}_0(\R^d)$ the space of smooth functions with compact support. For any positive integers $p,k \in \N$, $C^{k,p}_b := C^{k,p}_b([0,T] \times \R^d, \R)$ denotes the set of continuously differentiable bounded functions $[0,T] \times \R^d \rightarrow \R$ with uniformly bounded derivatives with respect to the time variable $t$ (resp. with respect to space variable $x$) up to order $k$ (resp. up to order $p$).  In particular, for $k = p = 0$, $C_b^{0,0}$ coincides with the space of bounded, continuous functions also denoted by $\shc_b$.
%$\mathcal{C}^{\infty}_b(\R^d)$ is the space of bounded and smooth functions. $\mathcal{C}_0(\R^d)$ denotes the space of continuous functions with compact support in %$\R^d$.
 For $r \in \N$, $W^{r,p}(\R^d)$ is the Sobolev space of order $r$ in $(L^p(\R^d),||\cdot||_{p})$, with $1 \leq p \leq \infty$. $W_{loc}^{1,1}(\R^d)$ denotes the space of functions $f : \R^d \rightarrow \R$ such that $f$ and $\nabla f$ (existing in the weak sense) belong to $L^1_{loc}(\R^d)$. \\
For convenience  we introduce the following notation. 
 \begin{itemize}   
 	\item $V\,:[0,T] \times {\mathcal C}^d \times \shc \times \shc^d $ is defined for any functions $x \in \shc^d$, $y \in \shc$ and $ z \in \shc^d$, by 
\begin{equation}
\label{eq:VP3}
V_t(x,y,z):=\exp \left ( \int_0^t \Lambda(s,x_s,y_s,z_s) ds\right )\quad \textrm{for any} \ t\in [0,T]\ .
\end{equation} 
\end{itemize}
The finite increments theorem gives, for all $(a,b) \in \R^2$,
\begin{eqnarray}
\label{eq:Vmajor}
\exp(a) - \exp(b) = (b-a) \int_0^1 \exp(\alpha a + (1-\alpha)b) d \alpha \ .
\end{eqnarray}
In particular, if $\Lambda$ is supposed to be bounded and Lipschitz w.r.t. to its space variables $(x,y,z)$, uniformly w.r.t. $t$, we observe that \eqref{eq:Vmajor} implies for all $t \in [0,T]$, $x,x' \in \shc^d$, $y,y' \in \shc$, $z,z' \in \shc^d$,
\begin{eqnarray}
\label{eq:LipV}
\vert V_t(x,y,z) - V_t(x',y',z') \vert \leq L_{\Lambda} e^{tM_{\Lambda}} \int_0^t \big( \vert x_s - x'_s \vert + \vert y_s - y'_s \vert + \vert z_s - z'_s \vert \big) ds \ ,
\end{eqnarray}
$M_{\Lambda}$ (resp. $L_{\Lambda}$) denoting an upper bound of $\vert \Lambda \vert$ (resp. the Lipschitz constant of $\Lambda$), see also Assumption 
\ref{ass:main2P3}. \\
In the whole paper, $(\Omega,\shf,(\shf_t)_{t \geq 0}, \P)$ will denote 
a filtered probability space and $W$ an $\R^p$-valued $(\shf_t)$-Brownian motion.

\subsection{Basic assumption}

We introduce here the basic assumption of the paper on Borel functions
$\Phi:[0,T] \times \R^d \rightarrow M_{d,p}(\R)$,  $g:[0,T] \times \R^d  \rightarrow \R^d$, and  
$\Lambda: [0,T] \times \R^d \times \R \times \R^d \rightarrow \R $ .
%and $g$ are functions  \rightarrow on $[0,T] \times \R^d$ taking values respectively in $M_{d,p}(\R)$ and $\R^d$.

\begin{ass} \label{ass:main2P3}
\begin{enumerate}
%%% VIEUX ITEMASSUMPTION 2
% 	\item   $\Phi$ and $g$ are functions defined on $[0,T] \times \R^d$ taking values respectively in $M_{d,p}(\R)$ and $\R^d$.
% There exist $\alpha \in ]0,1]$, $C_{\alpha}, L_{\Phi}, L_g > 0$, 
% such that for any $(t,t',x,x') \in [0,T] \times [0,T] \times \R^d \times \R^d$,
% 	$$
% 	\vert \Phi(t,x) - \Phi(t,x') \vert \leq C_{\alpha} \vert t-t' \vert^{\alpha} + L_{\Phi} \vert x-x' \vert \ ,
% 	$$
% 	and
% 	$$
% 	\vert g(t,x) - g(t',x') \vert \leq C_{\alpha} \vert t-t' \vert^{\alpha} + L_{g} \vert x-x' \vert \ .	
% 	$$
\item There exist positive reals $L_{\Phi}$, $L_g$ such that for any $(t,t',x,x') \in [0,T]^2 \times (\R^d)^2$,
	$$
	\vert \Phi(t,x) - \Phi(t,x') \vert \leq L_{\Phi} \big( \vert t-t' \vert^{\frac{1}{2}} + \vert x-x' \vert \big) \ ,
	$$
and 
	$$
	\vert g(t,x) - g(t,x') \vert \leq L_{g} \big( \vert t-t' \vert^{\frac{1}{2}} + \vert x-x' \vert \big) \ .
	$$
	\item $\Phi$ and $g$ belong to $C^{0,3}_b$. In particular, $\Phi$, $g$ are uniformly bounded and $M_\Phi$ (resp. $M_g$) denote the upper bound of $\vert \Phi \vert$ (resp. $\vert g \vert$).
    \item $\Phi$ is non-degenerate, i.e. there exists $c > 0$ such that for all $x \in \R^d$ 
    \begin{equation}  \label{def:NonDeg}
    \inf_{s\in[0,T]} \inf_{v \in \R^d \setminus \{0\}} \; \frac{\langle v,\Phi(s,x)\Phi^{t}(s,x)v 
    \rangle}{\vert v \vert^2} \geq c > 0.
    \end{equation}
%%%%%%%%%% OLD ITEM 4.
% 	\item $\Lambda$ is a Borel real-valued function defined on $[0,T]\times \R^d\times \R \times \R^d$ and Lipschitz uniformly w.r.t. $(t,x)$ i.e. 
% 	there exists a finite positive real, $L_{\Lambda}$, such that  for any $(t,x,z_1,z_1',z_2,z_2')\in [0,T] \times \R^d \times \R^2 \times (\R^d)^2$,
% we have
% \begin{eqnarray}
% \label{eq:LipAss}
% \vert \Lambda(t,x,z_1,z_2)-\Lambda(t,x,z_1',z_2')\vert \leq L_{\Lambda} ( \vert z_1-z_1' \vert + \vert z_2-z_2' \vert )\ .
% \end{eqnarray} 

\item There exists a positive real $L_{\Lambda}$, such that for any $(t,t',x,x',y,y',z,z') \in [0,T]^2 \times (\R^d)^2 \times \R^2 \times (\R^d)^2$,
	$$
	\vert \Lambda(t,x,y,z) - \Lambda(t',x',y',z') \vert \leq L_{\Lambda} \big( \vert t-t' \vert^{\frac{1}{2}} + \vert x-x' \vert + \vert y-y' \vert + \vert z-z' \vert\big) \ .
	$$
	\item $\Lambda$ is supposed to be uniformly bounded:  let $M_\Lambda$ be an upper bound for $\vert \Lambda \vert$.

\item $u_0$ is a Borel probability measure on $\R^d$ admitting a bounded density (still denoted by the 
same letter)  belonging to $W^{1,1}(\R^d)$.
% \item $(K_{\varepsilon})_{\varepsilon > 0}$ denotes a sequence of mollifiers explicitly given by $K_{\varepsilon}(x) := \frac{1}{\varepsilon^d}K(\frac{x}{\varepsilon})$, where the kernel $K$ is a probability density belonging to $W^{1,1}(\R^d) \cap W^{1,\infty}(\R^d)$.

\end{enumerate}
\end{ass}

% All items of Assumption \ref{ass:mainP3} are in force excepted 1. and 4. which are replaced by the following.
% \begin{enumerate}
% \item There exist positive reals $L_{\Phi}$, $L_g$ such that for any $(t,t',x,x') \in [0,T]^2 \times (\R^d)^2$,
% 	$$
% 	\vert \Phi(t,x) - \Phi(t,x') \vert \leq L_{\Phi} \big( \vert t-t' \vert^{\frac{1}{2}} + \vert x-x' \vert \big) \ ,
% 	$$
% and 
% 	$$
% 	\vert g(t,x) - g(t,x') \vert \leq L_{g} \big( \vert t-t' \vert^{\frac{1}{2}} + \vert x-x' \vert \big) \ .
% 	$$
% \setcounter{enumi}{3} 
% \item There exists a positive real $L_{\Lambda}$, such that for any $(t,t',x,x',y,y',z,z') \in [0,T]^2 \times (\R^d)^2 \times (\R^d)^2 \times (\R^d)^2$,
% 	$$
% 	\vert \Lambda(t,x,y,z) - \Lambda(t',x',y',z') \vert \leq L_{\Lambda} \big( \vert t-t' \vert^{\frac{1}{2}} + \vert x-x' \vert + \vert y-y' \vert + \vert z-z' \vert\big) \ .
% 	$$
% \end{enumerate}
% We also add the following item.
% \begin{enumerate}
% \setcounter{enumi}{6}
% \item $(K_{\varepsilon})_{\varepsilon > 0}$ denotes a sequence of mollifiers explicitly given by $K_{\varepsilon}(x) := \frac{1}{\varepsilon^d}K(\frac{x}{\varepsilon})$, where the kernel $K$ is a probability density belonging to $W^{1,1}(\R^d) \cap W^{1,\infty}(\R^d)$.
% \end{enumerate}
% \end{ass}

\subsection{Solution to the PDE}

\label{S22}
% %\setcounter{equation}{0}
% We first introduce the following assumption.
% \begin{ass}
% \label{ass:main0}
% \begin{enumerate}
% \item $\Phi$ and $g$ are functions defined on $[0,T] \times \R^d$ taking values in $M_{d,p}(\R^d)$ and $\R^d$. \\
% There exist $L_{\Phi}, L_g > 0$ such that for any $t \in [0,T]$, $(x,x') \in \R^d \times \R^d$,
% \begin{equation}
% \vert \Phi(t,x) - \Phi(t,x') \vert \leq L_{\Phi} \vert x-x' \vert \ ,
% \end{equation}
% and
% \begin{equation}
% \vert g(t,x) - g(t,x') \vert \leq L_{g} \vert x-x' \vert \ .
% \end{equation}
% \item The functions $s \in [0,T] \mapsto \vert \Phi(s,0) \vert$ and $s \in [0,T] \mapsto \vert g(s,0) \vert$ are bounded.
% \end{enumerate}
% \end{ass} 
In the whole paper we will write $a = \Phi \Phi^t$; in particular $a : [0,T] \times \R^d \longrightarrow \shs_d$.
%Through some definitions, we make here precise in which sense we will consider solutions of the PDE \eqref{eq:PDE}. We are interested in two different concepts of solutions $u : [0,T] \times \R^d \longrightarrow \R$ to that semilinear PDE where, for $t \in [0,T]$, 
Let $L_t$ be the second order partial differential operator such that 
% \begin{equation}
% \label{eq:PDE}
% \left \{
% \begin{array}{l}
% \partial_t u = L^{\ast}_t u + u\Lambda(t,x,u,\nabla_x u) \\
% u(0,\cdot) = u_0, \; \ ,
% \end{array}
% \right .
% \end{equation}
\begin{eqnarray}
\label{eq:generat}
(L_t \varphi)(x) = \frac{1}{2} \sum_{i,j=1}^d a_{i,j}(t,x) \partial_{ij}^2 \varphi(x) + \sum_{i=1}^d g_i(t,x) \partial_{i} \varphi(x), \; \varphi \in \shc_0^{\infty}(\R^d).
\end{eqnarray}
Its ''adjoint''  $L^{\ast}_t$ defined in  \eqref{eq:AdjGen2}, verifies
\begin{eqnarray}
\label{eq:AdjGen}
\int_{\R^d} L_t \varphi(x) \psi(x) dx = \int_{\R^d} \varphi(x) L_t^{\ast} \psi(x) dx \; , \; \varphi,\psi \in \shc_0^{\infty}(\R^d), t \in [0,T].
\end{eqnarray}
% and the initial condition $u_0$ of \eqref{eq:PDE} has to be understood in the sense that 
% $$
% \lim_{t \rightarrow 0} \int_{\R^d} \varphi(x)u_t(dx) = \int_{\R^d} \varphi(x)u_0(dx) \; , \textrm{ for all } \varphi \in \shc_0^{\infty}(\R^d) \ ,
% $$
% since, a priori, it can be irregular and not necessarily a function. \\
% More formally, $L^{\ast}_t$ can be written
% \begin{eqnarray}
% \label{eq:AdjGen2}
% L^{\ast}_t \psi = \frac{1}{2} \sum_{i,j=1}^d \partial_{ij}^2 \Big( a_{i,j}(t,\cdot) \psi \Big) + \sum_{i=1}^d \partial_{i} \Big( g_i(t,\cdot) \psi \Big), \; \psi \in \shc_0^{\infty}(\R^d) \ .
% \end{eqnarray}
%Let $\Lambda : [0,T] \times \R^d \times \R \times \R^d \longrightarrow \R$ be bounded, Borel measurable,
 We recall the notion of {\bf{weak solution}} 
% and {\bf{mild solution}} associated 
to \eqref{eq:PDE}.
\begin{defi}
\label{def:SolPDE}
Let $u : [0,T] \times \R^d \longrightarrow \R$ be a Borel function such that for every $t \in ]0,T]$, $u(t,\cdot)\in W^{1,1}_{\rm loc}(\R^d)$. 
%\begin{enumerate}
%\item 
$u$ will be called  {\bf{weak solution}} of \eqref{eq:PDE}
if for all $\varphi \in \shc_0^{\infty}(\R^d)$, $t \in [0,T]$,
\begin{eqnarray}
\label{eq:DefSolPDE}
\int_{\R^d} \varphi(x) u(t,x)dx - \int_{\R^d} \varphi(x) u_0(dx) & = & \int_0^t \int_{\R^d} u(s,x) L_s\varphi(x)dx ds \nonumber \\
&& + \; \int_0^t \int_{\R^d} \varphi(x) \Lambda(s,x,u(s,x),\nabla u(s,x)) u(s,x) dx ds \ . \nonumber
\end{eqnarray}
%\item
%$u$ will be called  {\bf{mild solution}} of \eqref{eq:PDE} if for all $\varphi \in \shc_0^{\infty}(\R^d)$, $t \in [0,T]$,
% \begin{eqnarray}
% \label{eq:DefMildSol}
% \int_{\R^d} \varphi(x) u(t,x)dx & = & \int_{\R^d} \varphi(x) \int_{\R^d} u_0(dx_0) P(0,x_0,t,dx)  \nonumber \\
% && + \; \int_{[0,t] \times \R^d} \Big( \int_{\R^d} \varphi(x) P(s,x_0,t,dx) \Big) \Lambda(s,x_0,u(s,x_0), \nabla u(s,x_0)) u(s,x_0)dx_0 ds  \ . \nonumber \\
% \end{eqnarray}
%\end{enumerate}
\end{defi}
We observe that when $\Lambda = 0$, \eqref{eq:PDE} is the classical Fokker-Planck equation.
%  which can be understood in the sense of distributions,
%  i.e., for all $\varphi \in \shc_0^{\infty}(\R^d)$, $t \in [0,T]$.
% \begin{eqnarray}
% \label{eq:SolFP}
% \int_{\R^d} u(t,dx) \varphi(x) = \int_{\R^d} u_0(dx) \varphi(x) + \int_0^t \int_{\R^d} u(s,dx) (L_s \varphi)(x) ds  \ .
% \end{eqnarray}

Theorem 3.6, Lemma 2.2, Remark 2.3 of \cite{LOR3}
allow to state the following.
\begin{thm} Under Assumption \ref{ass:main2P3} there exists
a unique weak solution of \eqref{eq:PDE} in  $L^{1}([0,T],W^{1,1}(\R^d)) \cap L^{\infty}([0,T] \times \R^d,\R)$.
\end{thm}

\subsection{Feynman-Kac type representation}
\label{SFK}
%We suppose here the validity of Assumption \ref{ass:main0}.

A weak solution of \eqref{eq:PDE} can be linked with a Feynman-Kac type equation, where we recall that a solution is given by a function $u : [0,T] \times \R^d \rightarrow \R$ satisfying the second line equation of \eqref{eq:NLSDE}. 

Let $Y_0$ be a random variable
 distributed according to $u_0$. Classical theorems for SDEs with Lipschitz coefficients imply, under Assumption \ref{ass:main2P3}, strong existence and pathwise uniqueness for the  SDE 
\begin{equation}
\label{eq:SDE}
dY_t = \Phi(t,Y_t) dW_t + g(t,Y_t)dt.
\end{equation}
%We fix  a random variable $Y_0$ distributed according to $u_0$ and consider the strong solution $Y$ of \eqref{eq:SDE}. \\

\begin{thm}
\label{thm:FKForm}
Assume that Assumption \ref{ass:main2P3} is fulfilled.
 We indicate by $Y$ the unique strong solution of \eqref{eq:SDE}. \\
%Suppose that $\Lambda : [0,T] \times \R^d \times \R \times \R^d \rightarrow \R$ is bounded and Borel measurable. DEJA INTRODUIT AUPARAVANT
%%%
%A weak solution $u : [0,T] \times \R^d \longrightarrow \R$ introduced in Section \ref{S22}  in $L^1([0,T],W^{1,1}(\R^d))$ 
Any real valued function $u  \in L^1([0,T],W^{1,1}(\R^d))$ 
is a weak solution of \eqref{eq:PDE} if and only if,
 for all 
$\varphi \in \shc_b(\R^d)$,
$t \in [0,T]$,
\begin{eqnarray}
\label{eq:FKForm}
\int_{\R^d} \varphi(x)u(t,x) dx = \E\Big[\varphi(Y_t) \exp \Big(\int_0^t \Lambda(s,Y_s,u(s,Y_s), \nabla u(s,Y_s)) \Big)   \Big] \ .
\end{eqnarray}
\end{thm}
\begin{rem} \eqref{eq:FKForm} will be called
 a {\bf{Feynman-Kac type representation}}
 of \eqref{eq:PDE}.
\end{rem}

\section{Particles system algorithm}
\label{SPartAlgo}

\setcounter{equation}{0}
% To simplify notations in the rest of the paper, $f_t$ will denote $f(t)$ 
% where $f : [0,T] \rightarrow E$ is an $E$-valued Borel function and 
% $(E,d_E)$ a metric space. \\
% In previous sections, we have studied existence, uniqueness 
% for a semilinear PDE of the form~\eqref{eq:PDE} and 
% we have established a Feynman-Kac type representation
%  for the corresponding solution $u$, see Theorem \ref{thm:FKForm}.
%, of the
% semilinear PDEs of the form~\eqref{eq:PDE}, see Theorem \ref{thm:FKForm}.
% The regularized form of \eqref{eq:PDE} is the integro-PDE \eqref{eq:PDEReg}
% for which we have established well-posedness in Proposition~\ref{C54} 1.
%  In the sequel, we denote by $\gamma^\varepsilon$ the corresponding solution
% and again by $u^{\varepsilon}(t,x) := (K_{\varepsilon} \ast \gamma^{\varepsilon}_t)(x)$ (see~\eqref{eq:ue}).
% %in particular $u^\varepsilon$ verifies equation~\eqref{eq:RFK}, see Proposition~\ref{C54} 2.
% We recall that $u^\varepsilon$  converges to $u$, when the regularization parameter
%  $\varepsilon$ vanishes to $0$, see Theorem \ref{thm:cvguu'}.
 In the present section, we propose a Monte Carlo approximation $u^{\varepsilon , N}$  of $u$,  providing an original numerical approximation of 
the semilinear PDE~\eqref{eq:PDE}, when both the number of particles $N\rightarrow \infty$ 
and the regularization parameter $\varepsilon \rightarrow 0$ with a judicious relative rate. 
%Let us fix a filtered probability space $(\Omega, \shf,(\shf_t)_{t \geq 0},\P)$ and $
%Let $u_0$ be a Borel probability measure on $\shp(\R^d)$.
% and \ref{ass:RegSection}. 
Let us consider a mollifier of the following form.
%Let us consider  $(K_{\varepsilon})_{\varepsilon > 0}$, a sequence of mollifiers 
%of the form

\begin{equation}
\label{eq:Kcond}
K\in W^{1,1}(\R^d)\cap W^{1,\infty}(\R^d)\ ,\quad \int_{\R^d} \vert x \vert^{d+1} \; K(x) dx<\infty\ ,\quad \textrm{and}\quad \int_{\R^d}\vert x\vert^{d+1} \; \vert \nabla K (x)\vert dx<\infty\ .
\end{equation}
We introduce the sequence of mollifiers, $(K_{\varepsilon})_{\varepsilon > 0}$,  explicitly given by 
\begin{equation}
\label{eq:Keps}
K_{\varepsilon}(x) := \frac{1}{\varepsilon^d} K\left(\frac{x}{\varepsilon}\right).
\end{equation} 
Obviously
\begin{equation}
\label{eq:HypK}
K_{\varepsilon} \xrightarrow[\varepsilon \rightarrow 0]{} \delta_0, \textrm{ (weakly)} \quad \textrm{and} \quad \forall \; \varepsilon > 0, K_{\varepsilon} \in W^{1,1}(\R^d) \cap W^{1,\infty}(\R^d) \ . 
\end{equation}

\subsection{Convergence of the particle system}
\label {S51}

%We suppose the validity of Assumption \ref{ass:mainP3}. \\
For fixed  $N \in \N^{\star}$, let $(W^i)_{i = 1,\cdots,N}$ be a family of independent Brownian motions and $(Y^i_0)_{i=1,\cdots,N}$ be i.i.d. random variables distributed according to $u_0$. For any $\varepsilon > 0$, we define the measure-valued functions  $(\gamma_t^{\varepsilon,N})_{t\in[0,T]}$ such that for any $t\in [0,T]$ 
\begin{equation}
\label{eq:XIiP3}
\left \{
\begin{array}{l}
\xi^{i}_t= \xi^{i}_0 +\int_0^t\Phi(s,\xi^{i}_s)dW^i_s+\int_0^tg(s,\xi^{i}_s)ds  \ , \quad \textrm{for}\ i=1,\cdots ,N\ ,\\
\xi^{i}_0 = Y^i_0 \quad \textrm{for}\ i=1,\cdots ,N\ ,\\
\gamma^{\varepsilon,N}_t = {\displaystyle \frac{1}{N}\sum_{i=1}^N  V_t\big (\xi^{i},(K_{\varepsilon} \ast \gamma^{\varepsilon,N})(\xi^{i}),(\nabla K_{\varepsilon} \ast \gamma^{\varepsilon,N})(\xi^{i}) \big ) } \delta_{\xi^{i}_t},
\end{array}
\right .
\end{equation}
%where $(K_{\varepsilon})_{\varepsilon > 0}$ denotes a sequence 
%of mollifiers such that for all $\varepsilon > 0$, $K_{\varepsilon} \in W^{1,1}(\R^d) \cap W^{1,\infty}(\R^d)$ and 
where we recall that $V_t$ is given by \eqref{eq:VP3}.
The first line of \eqref{eq:XIiP3} is  a $d$-dimensional classical SDE whose strong existence and pathwise uniqueness are ensured by classical theorems for Lipschitz coefficients.
 Clearly  $\xi^i, i=1,\cdots,N$ are  i.i.d. 
%In the following lemma, we prove by a fixed-point argument that the third line equation of \eqref{eq:XIiP3} has a unique solution. 
%e recall the object of Lemma 5.1 in \cite{LORTheoreticalPaper}.
%begin{lem}
%label{lem:PtFixuSn}
%e suppose the validity of Assumption \ref{ass:mainP3}.

The system \eqref{eq:XIiP3} is well-posed. Indeed let us fix $\varepsilon > 0$ and $N \in \N^{\star}$. Consider the i.i.d. system $(\xi^i)_{i=1,\cdots ,N}$ of particles, 
solution of the two first equations of~\eqref{eq:XIiP3}. 
By Lemma 5.1 of \cite{LOR3} 
we know  there exists a unique function
 $\gamma^{\varepsilon,N} : [0,T] \rightarrow \shm_f(\R^d)$ such that for all $t \in [0,T]$, $\gamma_t^{\varepsilon , N}$ is solution of~\eqref{eq:XIiP3}. 
Let us introduce 
%$u^{\varepsilon}$ be the real valued function defined by~\eqref{eq:ue}, and
 $u^{\varepsilon,N}$ such that for any $t \in [0,T]$, 
\begin{equation}
\label{eq:ueN}
u^{\varepsilon,N}(t,\cdot):=K_{\varepsilon}\ast \gamma_t^{\varepsilon,N} \ .
\end{equation}
%where $\gamma^{\varepsilon,N}$ is defined by the third line of~\eqref{eq:XIiP3}. 
Recalling Corollary 5.4 of \cite{LOR3}, $u^{\varepsilon,N}$ constitutes an approximation of $u$ solution of~\eqref{eq:PDE} in the following sense.
\begin{corro}
\label{thm:ThmCvg}
%Assume that the same assumptions as in Proposition \ref{prop:CvgParticle} are fulfilled. \\
Under Assumption~\ref{ass:main2P3}, 
there is a constant 
 $C$ (only depending on $M_{\Phi}$, $M_g$, $M_{\Lambda}$, $\Vert K \Vert_{\infty}$, 
$ \Vert \nabla K \Vert_{\infty}$, 
$L_{\Phi}$, $L_g$, $L_{\Lambda}$, $T$) 
such that the following holds. 
If $\varepsilon \rightarrow 0$, $N \rightarrow + \infty$ such that
 \begin{equation} \label{ERateConv}
\frac{1}{\sqrt{N \varepsilon^{d+4}}}e^{\frac{C}{\varepsilon^{d+1}}} \rightarrow 0,
\end{equation}
%for some constant $C$ (only depending on $M_{\Phi}$, $M_g$, $M_{\Lambda}$, $\Vert K \Vert_{\infty}$, 
%$ \Vert \nabla K \Vert_{\infty}$, 
%$L_{\Phi}$, $L_g$, $L_{\Lambda}$, $L_{\nabla K}$, $T$) 
%%% 
then
\begin{eqnarray} 
\E \Big[ \Vert u_t^{\varepsilon,N} - u_t \Vert_1 \Big] + \E \Big[ \Vert \nabla  u_t^{\varepsilon,N} - \nabla  u_t \Vert_1 \Big] \longrightarrow 0 \ .
\end{eqnarray}
\end{corro}
\begin{rem}
\label{rem:CvgPartic}
%Taking into account Theorem \ref{thm:cvguu'} above, it appears clearly that the convergence of $u^{\varepsilon,N}$ (resp. $\nabla u^{\varepsilon,N}$) to $u$ (resp. $\nabla u$) will hold as soon as $\frac{1}{\sqrt{N \varepsilon^{d+4}}}e^{\frac{C}{\varepsilon^{d+1}}} \rightarrow 0$ when $\varepsilon \rightarrow 0$, $N \rightarrow + \infty$. 
Condition~\eqref{ERateConv} constitutes a "trade-off" between the speed of 
convergence of $N$ and $\varepsilon$. Setting $\psi(\varepsilon) := \varepsilon^{-(d+4)} e^{\frac{2C}{\varepsilon^{d+1}}}$,
 that trade-off condition can be reformulated as 
 \begin{equation}
 \label{eq:TradeOff}
 %\frac{1}{\sqrt{N \varepsilon^{d+4}}}e^{\frac{C}{\varepsilon^{d+1}}} \rightarrow 0 \iff
 \frac{\psi(\varepsilon)}{N} \rightarrow 0 \quad \textrm{when} \quad
 \varepsilon \rightarrow 0, \; N \rightarrow + \infty.
 \end{equation}
 An example of such trade-off between $N$ and $\varepsilon$ can be given by the relation $\varepsilon (N) \propto (\frac{1}{log(N)})^{\frac{1}{d+4}}$. 
That type of tradeoff was obtained for instance in \cite{JourMeleard}, 
in the case of interacting particle system, without weighting function
$\Lambda$. However, we will observe that this theoretical sufficient
condition is far from being optimal. Indeed, in our simulations 
we observe that the classical tradeoff of kernel density estimates
based on i.i.d. random variables, i.e. 
$\varepsilon (N) \propto (\frac{1}{N})^{\frac{1}{d+4}}$ (see e.g. \cite{silverman1986}) seems to hold.
%%%% 
\end{rem}
% \begin{proof}
% Let us fix $\varepsilon > 0$, $N \in \N^{\star}$, $t \in [0,T]$. The proof is based on the  bound
% \begin{eqnarray}
% \E \Big[ \Vert u_t^{\varepsilon,N} - u_t \Vert_1 \Big] + \E \Big[ \Vert \nabla  u_t^{\varepsilon,N} - \nabla  u_t \Vert_1 \Big]& \leq & \E \Big[ \Vert u_t^{\varepsilon,N} - u_t^{\varepsilon} \Vert_1 \Big] + \E \Big[ \Vert \nabla u_t^{\varepsilon,N} - \nabla u_t^{\varepsilon} \Vert_1 \Big] \nonumber \\
% && \; + \Vert u_t^{\varepsilon} - u_t \Vert_1 + \Vert \nabla u_t^{\varepsilon} - \nabla u_t \Vert_1 \ , \nonumber \\
% & \leq & \frac{C}{\sqrt{N \varepsilon^{d+4}}}e^{\frac{C}{\varepsilon^{d+1}}} + \Vert u_t^{\varepsilon} - u_t \Vert_1 + \Vert \nabla u_t^{\varepsilon} - \nabla u_t \Vert_1 \ ,
% \end{eqnarray}
% where we have used Proposition \ref{prop:CvgParticle} for the second inequality above. Moreover, Theorem \ref{thm:cvguu'} gives $\Vert u_t^{\varepsilon} - u_t \Vert_1 + \Vert \nabla u_t^{\varepsilon} - \nabla u_t \Vert_1 \longrightarrow 0$ for $\varepsilon \rightarrow 0$. This concludes the proof of the corollary.
% \end{proof}
%%%%%%%%%%%% DEBUT DES NOUVELLES CONTRIBUTIONS
\subsection{Time discretized scheme}
% \label{SAlgo}
% In this section, we will make use of the assumptions below.
% \begin{ass}
% \label{ass:main2P3}
% All items of Assumption \ref{ass:mainP3} are in force excepted 1. and 4. which are replaced by the following.
% \begin{enumerate}
% \item There exist positive reals $L_{\Phi}$, $L_g$ such that for any $(t,t',x,x') \in [0,T]^2 \times (\R^d)^2$,
% 	$$
% 	\vert \Phi(t,x) - \Phi(t,x') \vert \leq L_{\Phi} \big( \vert t-t' \vert^{\frac{1}{2}} + \vert x-x' \vert \big) \ ,
% 	$$
% and 
% 	$$
% 	\vert g(t,x) - g(t,x') \vert \leq L_{g} \big( \vert t-t' \vert^{\frac{1}{2}} + \vert x-x' \vert \big) \ .
% 	$$
% %$\Phi$ and $g$ are functions defined on $[0,T] \times \R^d$ taking values respectively in $M_{d,p}(\R)$ and $\R^d$.
% % and are uniformly Holder continuous (exponent $0 < \alpha \leq 1$) w.r.t. $t$ and uniformly Lipschitz w.r.t. $x$ :
% \setcounter{enumi}{3} 
% \item There exists a positive real $L_{\Lambda}$, such that for any $(t,t',x,x',y,y',z,z') \in [0,T]^2 \times (\R^d)^2 \times (\R^d)^2 \times (\R^d)^2$,
% 	$$
% 	\vert \Lambda(t,x,y,z) - \Lambda(t',x',y',z') \vert \leq L_{\Lambda} \big( \vert t-t' \vert^{\frac{1}{2}} + \vert x-x' \vert + \vert y-y' \vert + \vert z-z' \vert\big) \ .
% 	$$
% \end{enumerate}
% We also add the following item.
% \begin{enumerate}
% \setcounter{enumi}{6}
% \item $(K_{\varepsilon})_{\varepsilon > 0}$ denotes a sequence of mollifiers explicitly given by $K_{\varepsilon}(x) := \frac{1}{\varepsilon^d}K(\frac{x}{\varepsilon})$, where the kernel $K$ is a probability density belonging to $W^{1,1}(\R^d) \cap W^{1,\infty}(\R^d)$.
% \end{enumerate}
% \end{ass}
We assume the validity of Assumption \ref{ass:main2P3}. For $n \in \N^{\star}$, we set $\delta t=T/n$ and introduce the time grid  $\big(0=t_0 < \cdots < t_k=k\delta t < \cdots < t_n=T \big)$. 
For any $N \in \N^{\star}, \,\varepsilon > 0$ and $n\in \N^\ast$, we define the measure-valued functions  $(\bar \gamma_t^{\varepsilon,N,n})_{t\in[0,T]}$ such that for any $t\in [0,T]$, 

%the real-valued function $K_{\varepsilon} : \R^d \rightarrow \R$ as mollification with bandwith parameter $\varepsilon$. 
%Using an elementary Euler scheme to discretize the time integrals, the time-discretized version $\big(\bar{\xi}, \bar{u}^{\varepsilon,N(\xi)} \big)$ of the continuous particle system $(\xi,u^{\varepsilon,N(\xi)})$ is given by
\begin{equation}
\label{eq:tildeYuP3}
\left \{
\begin{array}{l}
\bar \xi^{i}_{t}=\bar \xi^{i}_{0}+\int_0^t \Phi(r(s),\bar \xi^{i}_{r(s)})dW^i_{s}+\int_0^t g(r(s),\bar \xi^{i}_{r(s)})ds  \ ,\ \textrm{for}\ i=1,\cdots , N,\\
\bar{\xi}^{i}_0 = Y^i_0 \quad \textrm{for}\ i=1,\cdots ,N \ , \\
\displaystyle{ \bar{\gamma}^{\varepsilon,N,n}_{t}= \frac{1}{N} \sum_{i=1}^N \bar{V}_t \big( \bar{\xi}^i,(K_{\varepsilon} \ast \bar{\gamma}^{\varepsilon,N,n})(\bar{\xi}^i), (\nabla K_{\varepsilon} \ast \bar{\gamma}^{\varepsilon,N,n})(\bar{\xi}^i) \big) \delta_{\bar{\xi}_{t}^i} } \ , \\
%\bar{\gamma}^{\varepsilon,N,n}_{t}(x)=\frac{1}{N}\sum_{i=1}^N K_{\varepsilon}(x-\bar \xi^{i}_{t})\exp\Big \{\int_0^t \Lambda \big (r(s),\bar \xi^{i}_{r(s)},(K_{\varepsilon}\ast  \bar{\gamma}^{\varepsilon,N,n}_{r(s)})(\bar \xi^{i}_{r(s)}), (\nabla K_{\varepsilon}\ast  \bar{\gamma}^{\varepsilon,N,n}_{r(s)})(\bar \xi^{i}_{r(s)})\big )\,ds\Big \} \ ,\\
%\bar{\gamma}^{\varepsilon,N,n}_0 =  u_0 \ ,
\end{array}
\right .
\end{equation}
where for $(t,x,y,z) \in [0,T] \times \shc^d \times \shc \times \shc^d$,
\begin{eqnarray}
\label{eq:defVbar}
\bar{V}_t\big (x,y,z \big ) & := & \exp \Big \{\int_0^t \Lambda(r(s), x_{r(s)},y_{r(s)},z_{r(s)})\,ds \Big \} \ ,
\end{eqnarray}
and $r:\,s\in [0,T]\,\mapsto r(s)\in \{t_0,\cdots, t_n\}$ is the piecewise constant function such that $r(s)=t_k$ when $s\in [t_k,t_{k+1}[$. 
%In the sequel, we will often use the notation $\bar{\gamma}^{\varepsilon,N}_t$ instead of  $\bar{\gamma}^{\varepsilon,N,n}_t$ deliberately dropping the exponent $n$ to simplify the notation.
The proposition below establishes the convergence of the time discretized scheme \eqref{eq:tildeYuP3} to the continuous time version~\eqref{eq:XIiP3}. 
%for all $(t,x) \in [0,T] \times \R^d$ by 
%\begin{equation}
%\label{eq:ContVers}
%\left \{
%\begin{array}{l}
%\xi^{i}_t= \xi^{i}_0 +\int_0^t\Phi(s,\xi^{i}_s)dW^i_s+\int_0^tg(s,\xi^{i}_s)ds  \ , \quad \textrm{for}\ i=1,\cdots ,N\ ,\\
%\xi^{i}_0 = Y^i_0 \quad \textrm{for}\ i=1,\cdots ,N\ ,\\
%u^{\varepsilon,N}_t(x)={\displaystyle \frac{1}{N}\sum_{i=1}^N  K_{\varepsilon}(x-\xi^{i}_t) V_t\big (\xi^{i},u^{\varepsilon,N}(\xi^{i}),\nabla u^{\varepsilon,N}(\xi^{i}) \big ) } \ .
%\end{array}
%\right .
%\end{equation}
\begin{prop}
\label{prop:CvgTime}
Suppose the validity of Assumption \ref{ass:main2P3}.
In addition to condition~\eqref{eq:Kcond}, the gradient $\nabla K$ of $K$ is also supposed to be Lipschitz with the corresponding constant $L_{\nabla K}$.
%The first order partial derivatives of $K$ are also supposed to be Lipschitz. 
%and introduce the following notations: for all $(t,x) \in [0,T] \times \R^d$, $i \in \{1,\cdots,N\}$,
%\begin{eqnarray}
%\label{eq:DefsVi}
%V^i_t := V_t\big (\xi^{i},u^{\varepsilon,N}(\xi^{i}),\nabla u^{\varepsilon,N}(\xi^{i}) \big ) \quad \textrm{and} \quad \bar{V}^i_t := \bar{V}_t\big (\bar{\xi}^{i},\bar{u}^{\varepsilon,N}(\bar{\xi}^{i}),\nabla \bar u^{\varepsilon,N}(\bar{\xi}^{i}) \big ) \ ,
%\end{eqnarray}
%where for $(x,y,z) \in \shc^d \times \shc \times \shc^d$,
%\begin{eqnarray}
%\label{eq:defVbar}
%%V_t\big (x,y,z \big ) & := & \exp \Big \{ \int_0^t \Lambda(s,x_s,y_s,z_s) ds \Big \} \, \\
%\bar{V}_t\big (x,y,z \big ) & := & \exp \Big \{\int_0^t \Lambda(r(s), x_{r(s)},y_{r(s)},z_{r(s)})\,ds \Big \} \ .
%\end{eqnarray}
For  fixed parameters  $\varepsilon > 0$, $N \in \N^{\star}$ and $n\in \N^\star$, we introduce   $\bar u^{\varepsilon,N,n}$ such that for any $t \in [0,T]$, 
\begin{equation}
\label{eq:ueNn}
\bar u^{\varepsilon,N,n} (t,\cdot):=K_{\varepsilon}\ast \bar \gamma_t^{\varepsilon,N,n} \ ,
\end{equation} 
where $\bar \gamma_t^{\varepsilon,N,n}$ is defined by~\eqref{eq:tildeYuP3}.
%Let $u^{\varepsilon,N}$ (resp. $\bar u^{\varepsilon,N}$) be the function defined as in \eqref{eq:ContVers} (resp. \eqref{eq:tildeYu}) 
%with $K_{\varepsilon}$ satisfying item 7. of Assumption \ref{ass:main2}. 
%The time discretized particle system \eqref{eq:tildeYu} converges to the original particle system \eqref{eq:ContVers}. More precisely, 
%The following estimates hold:
%\begin{equation} \label{E537}
%\forall \; t \in [0,T], \; \E \Big[ \sup_{0 \leq t \leq T} \vert \xi^i_t - \bar \xi^i_t \vert^2 \Big] \leq \frac{\bar C}{n} \ ,
%\end{equation}
%and
Then 
%as the number of times steps $n$ grows to infinity,  the time discrete approximation $\bar u_t^{\varepsilon,N,n}$ converges to $u_t^{\varepsilon,N}$ defined by~\ref{eq:ueN}, for any $t\in [0,T]$, in the following sense
\begin{equation}
\label{eq:CvgTimeIneq}
\E \Big[ \Vert u^{\varepsilon,N}_t - \bar u^{\varepsilon,N,n}_t \Vert_1 \Big] +\E \Big[ \Vert \nabla u^{\varepsilon,N}_t - \nabla \bar u^{\varepsilon,N,n}_t \Vert_1 \Big]\leq \frac{\bar C}{\varepsilon^{d+3} \sqrt{n}} e^{\frac{\bar C}{\varepsilon^{d+1}}} \ ,
\end{equation}
where $\bar C$ is a finite, positive constant only depending on $M_{\Phi}$, $M_g$, $M_{\Lambda}$, $\Vert K \Vert_{\infty}$, 
$ \Vert \nabla K \Vert_{\infty}$, 
$L_{\Phi}$, $L_g$, $L_{\Lambda}$, $L_{\nabla K}$, $T$.
\end{prop}
%The result below finally states the convergence of $\bar{u}^{\varepsilon,N}_t$ to the solution $u$ of \eqref{eq:PDE} when $N,n \rightarrow + \infty$, $\varepsilon \rightarrow 0$.
From Proposition \ref{prop:CvgTime} and Corollary \ref{thm:ThmCvg} follows the result below.
\begin{thm}
\label{thm:Cvg} 
Suppose the validity of Assumption \ref{ass:main2P3}.
 In addition to condition~\eqref{eq:Kcond}, the gradient $\nabla K$ of $K$ is supposed to be Lipschitz with constant $L_{\nabla K}$.
Let $C, \bar C$ be the constants appearing in Corollary \ref{thm:ThmCvg}, equation~\eqref{ERateConv}
and Proposition~\ref{prop:CvgTime}, equation~\eqref{eq:CvgTimeIneq}. 
%The first order partial derivatives of $K$ are supposed to be Lipschitz.
%Let $(K_{\varepsilon})_{\varepsilon > 0}$ be a sequence of mollifiers given by \eqref{eq:Keps}, which fulfills the same assumptions as in Proposition \ref{prop:CvgTime}.
%Suppose moreover that
%\begin{eqnarray}
%\label{eq:MomentsCond}
%\int_{\R^d} \vert x \vert^{d+1} \; K(x) dx<\infty\ ,\quad \textrm{and}\quad \int_{\R^d}\vert x\vert^{d+1} \vert \; \nabla K (x)\vert dx<\infty\ .
%\end{eqnarray}
 %The kernel $K$ is supposed to be a density probability satisfying \eqref{eq:Kcond}. Assume also that the first order partial derivatives of $K$ are Lipschitz. 
%Let $\bar u^{\varepsilon,N,n}$ defined . 
%%Let $\bar u^{\varepsilon,N}$  be as defined in \eqref{eq:tildeYu}. 
%Let $u$ be the unique solution of \eqref{eq:PDE} or equivalently~\eqref{eq:FKForm}. \\
If $\varepsilon \rightarrow 0$, $n \rightarrow + \infty$ and $N \rightarrow + \infty$ such that
\begin{equation}
\label{eq:TradeCond}
\frac{1}{\sqrt{N \varepsilon^{d+4}}}e^{\frac{C}{\varepsilon^{d+1}}}
 \longrightarrow 0 \quad \textrm{and} \quad \frac{1}{\varepsilon^{d+3} \sqrt{n}} e^{\frac{\bar{C}}{\varepsilon^{d+1}}} \longrightarrow 0,
\end{equation}
%for some constants $C, \bar C$,
% where $C$ and $\bar{C}$ are constants respectively appearing in Proposition \ref{prop:CvgParticle} and \ref{prop:CvgTime},
% then for all $t \in [0,T]$,
then the particle approximation  $\bar u_t^{\varepsilon,N,n}$ defined by~\eqref{eq:ueNn} converges to the unique solution, $u$, of \eqref{eq:PDE}, 
%(or equivalently~\eqref{eq:FKForm}), 
in the  sense that for every $t$,
\begin{eqnarray}
\E \Big[ \Vert \bar u^{\varepsilon,N,n}_t - u_t \Vert_1 \Big] 
+\E \Big[ \Vert \nabla \bar u^{\varepsilon,N,n}_t - \nabla u_t \Vert_1 \Big]
\longrightarrow 0 \ .
\end{eqnarray}
%where $C$ and $\bar{C}$ are the constants appearing respectively in Proposition \ref{prop:CvgParticle} and Proposition \ref{prop:CvgTime}.
\end{thm}
\begin{proof}
For all $N,n \in \N^{\star}$, $\varepsilon > 0$ and $t \in [0,T]$, we have
\begin{eqnarray}
\label{eq:Decomp}
\E \Big[ \Vert \bar u^{\varepsilon,N,n}_t - u_t \Vert_1 \Big] 
+\E \Big[ \Vert \nabla \bar u^{\varepsilon,N,n}_t - \nabla u_t \Vert_1 \Big] 
& \leq & \E \Big[ \Vert \bar u^{\varepsilon,N,n}_t - u^{\varepsilon,N}_t \Vert_1 \Big] 
+\E \Big[ \Vert \nabla \bar u^{\varepsilon,N,n}_t - \nabla u^{\varepsilon,N}_t \Vert_1 \Big]\nonumber \\
&  &  + \E \Big[ \Vert u^{\varepsilon,N}_t - u_t \Vert_1 \Big] + \E \Big[ \Vert \nabla u^{\varepsilon,N}_t - \nabla u_t \Vert_1 \Big] 
\ .
\end{eqnarray}
Inequality \eqref{eq:CvgTimeIneq} of Proposition \ref{prop:CvgTime} and the second trade-off condition in \eqref{eq:TradeCond} imply that the first two expectations in the r.h.s. of \eqref{eq:Decomp} converges to $0$. \\
%We then emphasize that Assumption \ref{ass:main2} and
%By \eqref{eq:MomentsCond} and Remark \ref{rem:CvgPartic}.
By Corollary \ref{thm:ThmCvg},  the third and fourth expectations in the r.h.s. of \eqref{eq:Decomp} also converges to $0$. This concludes the proof.
%By Corollary \ref{thm:ThmCvg}, we claim it is enough to state that the third and fourth expectations in the r.h.s. of \eqref{eq:Decomp} also converges to $0$. This concludes the proof.
\end{proof}
The proof of Proposition \ref{prop:CvgTime} above will be based on the following technical lemma proved in the appendix.
\begin{lem}
\label{lem:TimeCvgineq}
We assume that the same assumptions as in Proposition \ref{prop:CvgTime} are fulfilled.
 %Let $\varepsilon > 0$, $(N,n) \in (\N^{\star})^2$ such that $\delta t = \frac{T}{n}$. %Let $u^{\varepsilon,N}$ (resp. $\bar %u^{\varepsilon,N}$) be
% the function defined by \eqref{eq:ueN} (resp. \eqref{eq:ueNn}). \\
Let $\bar u^{\varepsilon,N}$
 be the function, $\bar u^{\varepsilon,N,n}$, defined by \eqref{eq:ueNn}. \\
%for which we suppose that $K_{\varepsilon}$ is explicitly given by \eqref{eq:Keps}
%and $r : [0,T] \rightarrow [0,T]$ be a non-decreasing function such that $r(s) \leq s$ for all $s \in [0,T]$. \\
Then, there exists a constant $C > 0$, only depending on $M_{\Phi}$, $M_g$, $M_{\Lambda}$, $\Vert K \Vert_{\infty}$, 
$ \Vert \nabla K \Vert_{\infty}$, 
$L_{\Lambda}$, $L_{\nabla K}$ and $T$, such that for all $t \in [0,T]$,
$\varepsilon \in ]0,1]$, $n, N \in \N^*$ 
 the following estimates hold.
\begin{enumerate}
\item For almost all $x,y \in \R^d$,
\begin{equation}
\label{eq:Lem67Eq6}
\vert \bar u_{t}^{\varepsilon,N}(x) - \bar u_{t}^{\varepsilon,N}(y) \vert \leq \frac{C}{\varepsilon^{d+1}} \vert x - y \vert \quad \textrm{and} \quad \vert \nabla \bar u_{t}^{\varepsilon,N}(x) - \nabla \bar u_{t}^{\varepsilon,N}(y) \vert \leq \frac{C}{\varepsilon^{d+2}} \vert x - y \vert \ .
\end{equation}
\item 
\begin{equation}
\label{eq:Lem67Eq8}
 \E \Big[ \Vert \bar u^{\varepsilon,N}_t - \bar u^{\varepsilon,N}_{r(t)} \Vert_{\infty} \Big] \leq \frac{C \sqrt{\delta t}}{\varepsilon^{d+1}} \quad \textrm{and} \quad \E \Big[ \Vert \nabla \bar u^{\varepsilon,N}_t - \nabla \bar u^{\varepsilon,N}_{r(t)} \Vert_{\infty} \Big] \leq \frac{C \sqrt{\delta t}}{\varepsilon^{d+2}}\ ,
\end{equation}
where $\delta t := \frac{T}{n}$. 
\end{enumerate}
\end{lem}

\begin{proof}[Proof of Proposition \ref{prop:CvgTime}.]
In this proof, $C$ denotes a real positive constant 
(depending on $M_{\Phi}$, $M_g$, $M_{\Lambda}$, $\Vert K \Vert_{\infty}$,
$\Vert \nabla K \Vert_{\infty}$,
 $L_{\Phi}$, $L_g$, $L_{\Lambda}$, $L_{\nabla K}$, $T$) that may change from line to line. Let us fix $\varepsilon > 0$, $N \in \N^{\star}$, $n \in \N^{\star}$. \\
%Let us now prove inequality \eqref{eq:CvgTimeIneq}.
%Let us fix $\varepsilon > 0$, $N \in \N^{\star}$.
 For any $\ell=1,\cdots , d$, we introduce the real-valued function $G^{\ell}_{\varepsilon}$ defined on $\R^d$ such that 
\begin{equation}
\label{eq:G}
G^{\ell}_{\varepsilon}(x):=\frac{1}{\varepsilon^d} \frac{\partial K}{\partial x_{\ell}}\left(\frac{x}{\varepsilon}\right)\ ,\quad \textrm{for almost all }\ x\in \R^d\ .
\end{equation}
Let us now prove inequality \eqref{eq:CvgTimeIneq}.
% To this end, $G^{\ell}_{\varepsilon}$ will again denote the real-valued functions defined on $\R^d$ by
% \eqref{eq:G}. 
It is easy to observe that there exists a constant $C > 0$ depending on $\Vert K \Vert_1, \Vert \frac{\partial K}{\partial x_{\ell}} \Vert_1, \ell=1,\cdots,d,$ such that
\begin{equation}
\label{eq:KG}
\Vert K_{\varepsilon} \Vert_1 + \sum_{\ell=1}^d \Vert G^{\ell}_{\varepsilon} \Vert_1 \leq C \ ,
\end{equation}
and
\begin{equation}
\label{eq:KG2}
\Vert K_{\varepsilon} \Vert_{\infty} + \sum_{\ell=1}^d \Vert G^{\ell}_{\varepsilon} \Vert_{\infty} \leq \frac{C}{\varepsilon^d} \ .
\end{equation}
From \eqref{eq:ueN} and \eqref{eq:ueNn}, we recall that $u^{\varepsilon,N}$ and $\bar u^{\varepsilon,N}$ are defined by
\begin{equation}
\label{eq:uConvol}
\forall \; t \in [0,T], \; u_t^{\varepsilon,N} = K_{\varepsilon} \ast \gamma^{\varepsilon,N}_t \quad \textrm{and} \quad \bar u^{\varepsilon,N,n}_t = K_{\varepsilon} \ast \bar \gamma^{\varepsilon,N,n}_t.
\end{equation}
From now on we will set 
 $\bar u^{\varepsilon,N}:=\bar u^{\varepsilon,N, n}$ and
 $\bar \gamma^{\varepsilon,N}:=\bar \gamma^{\varepsilon,N, n}$.
For all $t \in [0,T]$, we have
\begin{eqnarray}
\label{eq:TimeCvg2}
\E \Big[ \Vert u^{\varepsilon,N}_t - \bar{u}^{\varepsilon,N}_{t} \Vert_1 \Big] + \E \Big[ \Vert \nabla u^{\varepsilon,N}_t - \nabla \bar{u}^{\varepsilon,N}_{t} \Vert_1 \Big]
& \leq & 
\E \Big[ \Vert K_{\varepsilon} \ast(\gamma^{\varepsilon,N}_t - \bar \gamma^{\varepsilon,N}_t) \Vert_1 \Big] + \frac{1}{\varepsilon} \sum_{l=1}^d \E \Big[ \Vert G^{\ell}_{\varepsilon} \ast (\gamma^{\varepsilon,N}_t - \bar \gamma^{\varepsilon,N}_t) \Vert_1 \Big] \nonumber \\
& \leq & \E \Big[ \Vert \gamma^{\varepsilon,N}_t - \bar \gamma^{\varepsilon,N}_t \Vert_{TV} \Big] + \frac{1}{\varepsilon}\sum_{\ell=1}^d \Vert G^{\ell}_{\varepsilon} \Vert_1 \E \Big[ \Vert \gamma^{\varepsilon,N}_t - \bar \gamma^{\varepsilon,N}_t \Vert_{TV} \Big] \nonumber \\
& = & \frac{C}{\varepsilon} \E \Big[ \Vert \gamma^{\varepsilon,N}_t - \bar \gamma^{\varepsilon,N}_t \Vert_{TV} \Big] \; \textrm{ by } \eqref{eq:KG} \ .
\end{eqnarray}
For $t \in [0,T]$, let us consider 
\begin{eqnarray}
\label{eq:TimeCvg3}
\E \Big[ \Vert \gamma^{\varepsilon,N}_t - \bar \gamma^{\varepsilon,N}_t \Vert_{TV} \Big] 
& = & 
\frac{1}{N} \sum_{i=1}^N  \E \Big[ \Big \vert V_t\big (\xi^{i},u^{\varepsilon,N}(\xi^{i}),\nabla u^{\varepsilon,N}(\xi^{i}) \big ) - \bar{V}_t\big (\bar{\xi}^{i},\bar{u}^{\varepsilon,N}(\bar{\xi}^{i}),\nabla \bar u^{\varepsilon,N,n}(\bar{\xi}^{i}) \big ) \Big \vert \Big] \nonumber \\
& \leq & 
\frac{1}{N} \sum_{i=1}^N  \E \Big[ \Big \vert V_t\big (\xi^{i},u^{\varepsilon,N}(\xi^{i}),\nabla u^{\varepsilon,N}(\xi^{i}) \big ) - V_t\big (\xi^{i},\bar{u}^{\varepsilon,N}(\xi^{i}),\nabla \bar u^{\varepsilon,N}(\xi^{i}) \big ) \Big \vert \Big] \nonumber \\
&& 
\; + \; \frac{1}{N} \sum_{i=1}^N  \E \Big[ \Big \vert V_t\big (\xi^{i},\bar u^{\varepsilon,N}(\xi^{i}),\nabla \bar u^{\varepsilon,N}(\xi^{i}) \big ) - V_t\big (\bar \xi^{i},\bar{u}^{\varepsilon,N}(\bar \xi^{i}),\nabla \bar u^{\varepsilon,N}(\bar \xi^{i}) \big ) \Big \vert \Big] \nonumber \\
&& 
\; + \; \frac{1}{N} \sum_{i=1}^N  \E \Big[ \Big \vert V_t\big (\bar \xi^{i},\bar u^{\varepsilon,N}(\bar \xi^{i}),\nabla \bar u^{\varepsilon,N}(\bar \xi^{i}) \big ) - \bar V_t\big (\bar \xi^{i},\bar{u}^{\varepsilon,N}(\bar \xi^{i}),\nabla \bar u^{\varepsilon,N}(\bar \xi^{i}) \big ) \Big \vert \Big]. \nonumber \\
%& \leq & \frac{1}{N} \sum_{i=1}^N A^{i,\varepsilon,N,n}_t + B^{i,\varepsilon,N,n}_t + C^{i,\varepsilon,N,n}_t \ ,
\end{eqnarray}
We are now interested in bounding each term in the r.h.s. of \eqref{eq:TimeCvg3}. Let us fix $t \in [0,T]$, $i \in \{1,\cdots,N\}$. Since $\Lambda$ is bounded and Lipschitz, inequality \eqref{eq:LipV} implies
\begin{eqnarray}
\label{eq:ViMajor}
A^{i,\varepsilon,N,n}_t & := & \E \Big[ \Big \vert V_t\big (\xi^{i},u^{\varepsilon,N}(\xi^{i}),\nabla u^{\varepsilon,N}(\xi^{i}) \big ) - V_t\big (\xi^{i},\bar{u}^{\varepsilon,N}(\xi^{i}),\nabla \bar u^{\varepsilon,N}(\xi^{i}) \big ) \Big \vert \Big] \nonumber \\
& \leq & e^{M_{\Lambda}T} \E \Big[ \int_0^t \Big \vert \Lambda(s, \xi^{i}_{s},u^{\varepsilon,N}_{s}( \xi^{i}_{s}), \nabla u^{\varepsilon,N}_{s}(\xi^{i}_{s})) - \Lambda(s, \xi^{i}_{s},\bar{u}^{\varepsilon,N}_{s}(\xi^{i}_{s}), \nabla \bar{u}^{\varepsilon,N}_{s}( \xi^{i}_{s}))  \Big \vert \Big] \; ds \nonumber \\
& \leq & e^{M_{\Lambda}T} L_{\Lambda} \int_0^t \Big \{ \E \big[ \vert u_s^{\varepsilon,N}(\xi^{i}_s) - \bar{u}_{s}^{\varepsilon,N}(\xi^{i}_{s}) \vert \big] + \E \big[ \vert \nabla u_s^{\varepsilon,N}(\xi^{i}_s) - \nabla \bar{u}_{s}^{\varepsilon,N}(\xi^{i}_{s}) \vert \big]    \Big \} \; ds \ . \nonumber \\
\end{eqnarray}
Taking into account \eqref{eq:uConvol}, for all $s \in [0,T]$, it follows
\begin{eqnarray}
\label{eq:Major1}
\E \big[ \vert u_s^{\varepsilon,N}(\xi^{i}_s) - \bar{u}_{s}^{\varepsilon,N}(\xi^{i}_{s}) \vert \big] & = & \E \big[ \vert K_{\varepsilon} \ast (\gamma^{\varepsilon,N}_s - \bar \gamma^{\varepsilon,N}_s )(\xi^{i}_{s}) \vert \big] \nonumber \\
& \leq & 
%OLD \frac{C}{\varepsilon^{d+1}}
  \frac{C}{\varepsilon^{d}}
 \E \Big[ \Vert \gamma^{\varepsilon,N}_s - \bar \gamma^{\varepsilon,N}_s \Vert_{TV} \Big] \ ,
\end{eqnarray}
where we have used inequality \eqref{eq:KG2}. Similarly, we also obtain
\begin{eqnarray}
\label{eq:Major2}
\E \big[ \vert \nabla u_s^{\varepsilon,N}(\xi^{i}_s) - \nabla \bar{u}_{s}^{\varepsilon,N}(\xi^{i}_{s}) \vert \big] & = & \frac{1}{\varepsilon} \sum_{\ell=1}^d \E \big[ \vert G_{\varepsilon}^{\ell} \ast (\gamma^{\varepsilon,N}_s - \bar \gamma^{\varepsilon,N}_s )(\xi^{i}_{s}) \vert \big] \nonumber \\
& \leq & \frac{C}{\varepsilon^{d+1}} \E \Big[ \Vert \gamma^{\varepsilon,N}_s - \bar \gamma^{\varepsilon,N}_s \Vert_{TV} \Big] \ ,
\end{eqnarray}
for all $s \in [0,T]$. Injecting \eqref{eq:Major1} and \eqref{eq:Major2} in the r.h.s. of \eqref{eq:ViMajor} yields
\begin{eqnarray}
\label{eq:MajorA}
A^{i,\varepsilon,N,n}_t & \leq & \frac{C}{\varepsilon^{d+1}} \int_0^t \E \Big[ \Vert \gamma^{\varepsilon,N}_s - \bar \gamma^{\varepsilon,N}_s \Vert_{TV} \Big] \; ds \ .
\end{eqnarray}
%Concerning the term $B^{i,\varepsilon,N,n}_t$, 
Concerning the second term in the r.h.s. of \eqref{eq:TimeCvg3}, we invoke again \eqref{eq:LipV} to obtain
\begin{eqnarray}
\label{eq:MajorB}
B^{i,\varepsilon,N,n}_t & := & \E \Big[ \Big \vert V_t\big (\xi^{i},\bar u^{\varepsilon,N}(\xi^{i}),\nabla \bar u^{\varepsilon,N}(\xi^{i}) \big ) - V_t\big (\bar \xi^{i},\bar{u}^{\varepsilon,N}(\bar \xi^{i}),\nabla \bar u^{\varepsilon,N}(\bar \xi^{i}) \big ) \Big \vert \Big] \nonumber \\
& \leq & e^{M_{\Lambda}T}L_{\Lambda} \E \Big[ \int_0^t \Big \vert \Lambda(s, \xi^{i}_{s},\bar u^{\varepsilon,N}_{s}( \xi^{i}_{s}), \nabla \bar u^{\varepsilon,N}_{s}(\xi^{i}_{s})) - \Lambda(s, \bar \xi^{i}_{s},\bar{u}^{\varepsilon,N}_{s}(\bar \xi^{i}_{s}), \nabla \bar{u}^{\varepsilon,N}_{s}( \bar \xi^{i}_{s}))  \Big \vert \Big] \; ds \nonumber \\
& \leq & e^{M_{\Lambda}T} L_{\Lambda} \int_0^t \Big \{ \E \Big[ \vert \xi^i_s - \bar \xi^i_s \vert \Big] + \E \big[ \vert \bar u_s^{\varepsilon,N}(\xi^{i}_s) - \bar{u}_{s}^{\varepsilon,N}(\bar \xi^{i}_{s}) \vert \big] + \E \big[ \vert \nabla \bar u_s^{\varepsilon,N}(\xi^{i}_s) - \nabla \bar{u}_{s}^{\varepsilon,N}(\bar \xi^{i}_{s}) \vert \big]    \Big \} \; ds \nonumber \\
& \leq & \frac{Ce^{M_{\Lambda}T} L_{\Lambda}T \sqrt{\delta t}}{\varepsilon^{d+2}} \nonumber \\
& \leq & \frac{C}{\varepsilon^{d+2} \sqrt{n}}, %\frac{C \sqrt{\delta t}}{\varepsilon^{d+2}} \ ,
\end{eqnarray}
where we have used successively classical bounds of the Euler scheme (see e.g. Section 10.2, Chapter 10 in \cite{kloeden})  and \eqref{eq:Lem67Eq6}. \\
Regarding the third term, similarly as for the above inequality 
\eqref{eq:MajorB},
\eqref{eq:LipV} yields
\begin{eqnarray}
\label{eq:MajorC}
C^{i,\varepsilon,N,n}_t & := & \E \Big[ \Big \vert V_t\big (\bar \xi^{i},\bar u^{\varepsilon,N}(\bar \xi^{i}),\nabla \bar u^{\varepsilon,N}(\bar \xi^{i}) \big ) - \bar V_t\big (\bar \xi^{i},\bar{u}^{\varepsilon,N}(\bar \xi^{i}),\nabla \bar u^{\varepsilon,N}(\bar \xi^{i}) \big ) \Big \vert \Big] \nonumber \\
& \leq & e^{M_{\Lambda}T}L_{\Lambda} \int_0^t \Big( \vert s - r(s) 
\vert^{\frac{1}{2}} + \E \Big[ \vert \bar{\xi}^i_s - \bar{\xi}^i_{r(s)} \vert \Big] + \E \Big[ \vert \bar u^{\varepsilon,N}_{s}(\bar{\xi}^{i}_s) -  \bar u^{\varepsilon,N}_{r(s)}(\bar \xi^{i}_{r(s)})  \vert  \Big] \nonumber \\
&& \; + \; \E \Big[ \vert \nabla \bar u^{\varepsilon,N}_{s}(\bar{\xi}^{i}_s) -  \nabla \bar u^{\varepsilon,N}_{r(s)}(\bar \xi^{i}_{r(s)})  \vert  \Big]  \Big) ds \ ,
\end{eqnarray}
where we have used H\"older property of $\Lambda$ w.r.t. the time variable. \\
%From the definition of the piecewise-constant function $r$, it is clear that for all $s \in [0,T]$,
%\begin{equation}
%\label{eq:grid}
%\vert s - r(s) \vert \leq \delta t .
%\end{equation}
Boundedness of $\Phi$, $g$ with classical Burkholder-Davis-Gundy (BDG) inequality give 
\begin{eqnarray}
\label{eq:648}
\E \big[ \vert \bar \xi^i_s - \bar \xi^i_{r(s)} \vert \big] & \leq & 2C \sqrt{ \delta t} \leq \frac{C}{\sqrt{n}} \ , \quad s \in [0,T] \ .
\end{eqnarray}
To bound the third term in the r.h.s. of \eqref{eq:MajorC}, we use the following decomposition: for all $s \in [0,T]$, 
\begin{eqnarray}
\label{eq:659}
\E \big[ \vert \bar u_s^{\varepsilon,N}(\bar \xi^{i}_s) - \bar{u}_{r(s)}^{\varepsilon,N}(\bar{\xi}^{i}_{r(s)}) \vert \big] & \leq & \E \big[ \vert \bar u_s^{\varepsilon,N}(\bar \xi^{i}_s) - \bar{u}_{s}^{\varepsilon,N}(\bar \xi^{i}_{r(s)}) \vert \big] + \E \big[ \vert \bar u_s^{\varepsilon,N}(\bar \xi^{i}_{r(s)}) - \bar{u}_{r(s)}^{\varepsilon,N}(\bar{\xi}^{i}_{r(s)}) \vert \big] \ .
\end{eqnarray}
We first observe that the first inequality \eqref{eq:Lem67Eq6} %with $r(s) = s$
%ANTHONY. ATTENTION $r(s)$ EST FIXE. NE DEVRAIT-ON PAS DEFINIR $r(s)$ COMME UNE FONCTION 
%GENERIQUE, VOIR UN DE NOS PAPIERS PRECEDENTS?
gives
\begin{eqnarray}
\label{eq:660}
\E \big[ \vert \bar u_s^{\varepsilon,N}(\bar \xi^{i}_s) - \bar{u}_{s}^{\varepsilon,N}(\bar \xi^{i}_{r(s)}) \vert \big] & \leq & \frac{C}{\varepsilon^{d+1}} \E \Big[ \vert \bar \xi^i_s - \bar \xi^i_{r(s)} \vert \Big] \leq \frac{C \sqrt{\delta t}}{\varepsilon^{d+1}} \leq \frac{C }{\varepsilon^{d+1} \sqrt{n}} \ ,
\end{eqnarray}
for all $s \in [0,T]$. Invoking now the first inequality of \eqref{eq:Lem67Eq8} leads to
\begin{equation}
\label{eq:661}
\E \big[ \vert \bar u_s^{\varepsilon,N}(\bar \xi^{i}_{r(s)}) - \bar{u}_{r(s)}^{\varepsilon,N}(\bar{\xi}^{i}_{r(s)}) \vert \big] \leq \frac{C \sqrt{\delta t}}{\varepsilon^{d+1}} \leq \frac{C}{\varepsilon^{d+1} \sqrt{n}} \ , \quad s \in [0,T] \ .
\end{equation}
Injecting now \eqref{eq:661} and \eqref{eq:660} in \eqref{eq:659} yield
\begin{equation}
\label{eq:662}
\E \big[ \vert \bar u_s^{\varepsilon,N}(\bar \xi^{i}_s) - \bar{u}_{r(s)}^{\varepsilon,N}(\bar{\xi}^{i}_{r(s)}) \vert \big] \leq 
%\frac{C \sqrt{\delta t}}{\varepsilon^{d+1}} \leq
 \frac{C}{\varepsilon^{d+1} \sqrt{n}} \ , \quad s \in [0,T] \ .
\end{equation}
With very similar arguments as those used to obtain \eqref{eq:662} (i.e. decomposition \eqref{eq:659} and inequalities \eqref{eq:Lem67Eq6}, \eqref{eq:Lem67Eq8}), we obtain for all $s \in [0,T]$, 
\begin{equation}
\label{eq:663}
\E \Big[ \vert \nabla \bar u^{\varepsilon,N}_{s}(\bar{\xi}^{i}_s) -  \nabla \bar u^{\varepsilon,N}_{r(s)}(\bar \xi^{i}_{r(s)})  \vert  \Big] \leq \frac{C \sqrt{\delta t}}{\varepsilon^{d+2}} \leq \frac{C}{\varepsilon^{d+2} \sqrt{n}} \ .
\end{equation}
Gathering \eqref{eq:663}, \eqref{eq:662} and \eqref{eq:648} in \eqref{eq:MajorC} gives
\begin{eqnarray}
\label{eq:664}
C^{i,\varepsilon,N,n}_t & \leq & \frac{C \sqrt{\delta t}}{\varepsilon^{d+2}} \leq \frac{C}{\varepsilon^{d+2} \sqrt{n}}\ .
\end{eqnarray}
Finally, injecting \eqref{eq:664}, \eqref{eq:MajorB} and \eqref{eq:MajorA} in \eqref{eq:TimeCvg3}, we obtain for all $t \in [0,T]$, 
\begin{eqnarray}
\E \Big[ \Vert \gamma^{\varepsilon,N}_t - \bar \gamma^{\varepsilon,N}_t \Vert_{TV} \Big] & \leq & C \Big( \frac{1}{\varepsilon^{d+2} \sqrt{n}}   + \frac{1}{\varepsilon^{d+1}} \int_0^t \E \Big[ \Vert \gamma^{\varepsilon,N}_s - \bar \gamma^{\varepsilon,N}_s \Vert_{TV} \Big] \; ds \Big) \ .
\end{eqnarray}
Gronwall's lemma applied to the function $t \in [0,T] \mapsto \E \Big[ \Vert \gamma^{\varepsilon,N}_t - \bar \gamma^{\varepsilon,N}_t \Vert_{TV} \Big]$ implies
\begin{equation}
\label{eq:Final}
\E \Big[ \Vert \gamma^{\varepsilon,N}_t - \bar \gamma^{\varepsilon,N}_t \Vert_{TV} \Big] \leq \frac{C}{\varepsilon^{d+2} \sqrt{n}} e^{\frac{C}{\varepsilon^{d+1}}} \ , \quad t \in [0,T] \ .
\end{equation}
The result follows by injecting \eqref{eq:Final} in \eqref{eq:TimeCvg2}.
\end{proof}

The particle algorithm used to simulate the dynamics~\eqref{eq:tildeYuP3} consists of the following steps.
\begin{description}
	\item[Initialization] 
	for $k=0$.
	\begin{enumerate} 
		\item Generate $(\bar{\xi}_{0}^{i})_{i=1,..,N}$  i.i.d.$\sim$ $u_0(x)dx$;
		\item set $G_0^i := 1$, $i = 1,\cdots,N$;
		\item set $\bar{u}^{\varepsilon,N}_{t_0}(\cdot) := (K_{\varepsilon} \ast u_0)(\cdot)$.   
	\end{enumerate}
	\item[Iterations]  
	for $k = 0, \cdots , n-1$.
		\begin{itemize}
			\item For $i=1,\cdots N$, set 
			$\bar \xi^{i}_{t_{k+1}}:= {\bar{\xi}^{i}_{t_k}} + \Phi(t_{k},{\bar{\xi}^{i}_{t_k}}) \sqrt{\delta t} \; \epsilon^i_{k+1} + g(t_{k},{\bar{\xi}^{i}_{t_k}}) \delta t \ ,$
			where $(\epsilon^i_{k})^{i=1, \cdots, N}_{k=1,\cdots n}$ is a sequence of i.i.d centered and standard Gaussian variables; 
		\end{itemize}
 
		\begin{itemize}
			\item for $i=1,\cdots N$, set $G_{k+1}^i := G_{k}^i \times \exp \left( \Lambda(t_{k},{ \bar{\xi}^{i}_{t_k}},{\bar{u}}^{\varepsilon,N}_{t_k}({ \bar{\xi}^{i}_{t_k}}), \nabla {\bar{u}}^{\varepsilon,N}_{t_k}({ \bar{\xi}^{i}_{t_k}}) ) \delta t  \right)\ ;$
			\item set ${\bar{u}}^{\varepsilon,N}_{t_{k+1}}(\cdot) = {\displaystyle \frac{1}{N} \sum_{i=1}^{N}  G_{k+1}^i \times K_{\varepsilon}(\cdot - { \bar{\xi}^{i}_{t_{k+1}}})}. $        
  		\end{itemize}
\end{description}

\begin{rem}
Observe that each particle evolves independently without any  interaction by contrast to the case considered in \cite{LOR1,LOR2}. However, since the evaluation of the function ${\bar{u}}^{\varepsilon,N}$ at any point $(t_k,\bar \xi^i_{t_k})$ requires to sum up $N$ terms, the complexity of the algorithm is still of order $nN^2$. 
However, there are several strategies to speed up the evaluation of ${\bar{u}}^{\varepsilon,N}_{t_k}(\bar \xi^i_{t_k})$. 
By a judicious partition of the space, we can efficiently approximate this evaluation with a complexity of order $N\log(N)$. 
The basic idea is that, only a small part of the particles will really contribute to ${\bar{u}}^{\varepsilon,N}_{t_k}(\bar \xi^i_{t_k})$, most of particles being too far away from $\bar \xi^i_{t_k}$. Dual tree recursions based on k-d tree allow to perform this approximation efficiently with tight accuracy guarantees, see~\cite{GrayMoore2003}.

\end{rem}

\section{Numerical simulations}
\label{sec:simulations}
\setcounter{equation}{0}

The aim of this section is to illustrate the performances  of our original
 numerical scheme  to approximate the solution 
of semilinear PDEs~\eqref{eq:PDE}, inspect to what extent this approach
 remains valid out of  Assumption~\ref{ass:main2P3}
and to provide a perspective of application to stochastic control problems.
% Following~\cite{DelarueMenozzi}, 
First we consider the one dimensional Burgers equation and then
the production / inventory control problem that we relate
to
%We consider two types of semilinear PDEs, for which a semi-explicit expression of the solution is available: the one dimensional Burgers equation and 
the $d$-dimensional KPZ equation.

\subsection{Burgers equation}

Let $u_0$ be a probability density on $\R$ and set
 $U_0 = \int^\cdot_{-\infty} u_0(y) dy$. 
Let us consider the viscid Burgers equation in dimension $d = 1$, given by
\begin{equation}
\label{eq:Burgers}
\left \{ 
\begin{array}{l}
\partial_t u = \frac{\nu^2}{2} \partial_{xx} u - u \partial_x u, \quad (t,x) \in [0,T] \times \R, \nu > 0 \\
u(0,\cdot) = u_0 \ .
\end{array}
\right .
\end{equation}
It is well-known (see e.g. \cite{DelarueMenozzi}) that \eqref{eq:Burgers} admits a unique classical solution 
if $u_0 \in L^1(\R^d)$.
Moreover, 
using the so-called Cole-Hopf transformation, the solution $u$ admits the semi-explicit formula
\begin{eqnarray}
\label{eq:ExpliBurg}
u(t,x) = \frac{\E[u_0(x+ \nu B_t)e^{-\frac{U_0(x+ \nu B_t)}{\nu^2}}]}{\E[e^{-\frac{U_0(x+\nu B_t)}{\nu^2}}]}, \quad (t,x) \in [0,T] \times \R  ,
\end{eqnarray}
%where $u_0$ and $U_0$ are linked by the relation $u_0 = \frac{d}{dx} U_0$ and
where $B$ denotes the real-valued standard Brownian motion. 
%In fact \eqref{eq:Burgers} constitutes  a semi-explicit expression of
%the solution. 
Integrating against test functions in space it is not difficult to show that
the classical solution $u$ is also a weak solution of \eqref{eq:PDE} with
$$\Phi = \nu, g \equiv 0, \Lambda(t,x,y,z) =  z.$$ Apparently our Assumption \ref{ass:main2P3} is not fulfilled,
at least for what concerns $\Lambda$.
However choosing $u_0$ being a bounded probability density, it is not difficult to show that
there exists $M > 0$ such that $u$ is a solution of the subsidiary equation of type \eqref{eq:PDE}
with $\Phi \equiv \nu, \Lambda(t,x,y,z) := \Lambda_M (z)$ where $\Lambda_M: \R \rightarrow \R$
is a smooth bounded function such that $\Lambda_M(z) = z$ if $\vert z \vert \le M$
and  $\Lambda_M(z) = 0$  if      $\vert z \vert > M + 1  $.
In this case  Assumption \ref{ass:main2P3} is fulfilled for the subsidiary equation.

In our numerical tests, we have implemented 
the time discretized particle scheme~\eqref{eq:tildeYuP3}  with the following values of parameters 
$\Phi(t,x) := \nu ,\,
g(t,x) := 0 ,\,
\Lambda(t,x,y,z) := z\,$, in order to approximate the solution of~\eqref{eq:Burgers}. 
%Observe that $\Lambda$ is indeed Lipschitz but unbounded. 

\subsection{The production/inventory control problem
and KPZ (deterministic) equation}

Let us introduce a multivariate  extension of the Production/Inventory planning studied in~\cite{Bensoussanetal84}. 
Consider a factory producing several goods indexed by $i=1,\cdots ,d$. For each good $i$ and any time $t\in [0,T]$, let 
$(X^i_t)$ denote the inventory
level; $(D^i_t)$  the random demand rate and $(p^i_t)$ the production rate at time $t$.  Let us denote $X_t:=(X^i_t)_{i=1,\cdots , d}\,$, $p_t:=(p^i_t)_{i=1,\cdots , d}$ and $D_t:=(D^i_t)_{i=1,\cdots , d}$. The $d$-dimensional inventory process $X$ is modelled as the controlled
diffusion
\begin{equation}
\label{eq:inventoryDyn}
\left\{
\begin{array}{l}
dX^{0,x,p}_t = p_tdt-dD_t\ ,\quad \textrm{with}\quad dD_t=\bar d_t dt + diag(\sigma)dW_t\\
X^{0,x,p}_0=x\ ,
\end{array}
\right .
\end{equation}
where  $W$ is a $d$-dimensional Brownian motion, $\bar d_t\in \R^d$ is the (deterministic) average demand rate and $\sigma=(\sigma_1,\cdots,\sigma_d)$ with $\sigma_i$ being the volatility of the demand rate $D^i$. The aim is to minimize over non-anticipative production rates $(p_t)$, the following expected cost: 
\begin{equation}
\label{eq:inventoryCost}
\E\left [g(X_T)+\int_0^T \big [\sum_{i=1}^d c^i(p^i_s-\bar p^i_s)^2+h(X_s)\big ]\,ds\right ]\ ,
\end{equation}
where $(c^i)_i$ and $(\bar p^i)_{i}$ are parameters for the quadratic production cost and $h,g: x\in\R^d\mapsto h(x), g(x)\in\R$ are  nonlinear functions 
respectively representing the inventory holding cost and the inventory terminal cost.
The value function is 
%By the Dynamical Programming Principle, the value function associated to this control problem
\begin{equation}
\label{eq:Bellman}
v(t,x):=\sup_{p}\E\left [g(X^{t,x}_T)+\int_t^T \big [\sum_{i=1}^d c^i(p^i_s-\bar p^i_s)^2+h(X^{t,x}_s)\big ]\,ds\right ].
\end{equation}
$v$
is solution of the
 Hamilton-Jacobi-Bellman equation 
\begin{equation}
\label{eq:HJB}
\left \{
\begin{array}{l}
\partial_tv+ \sum_{i=1}^d \frac{1}{4c_i}(\partial_{x_i}v)^2+\sum_{i=1}^d (\bar p^i_t-\bar{d}^i_t)\partial_{x_i}v+\frac{1}{2}\sum_{i=1}^d \sigma^2_i \partial^2_{x_ix_i}v-h=0\\
v(T,x)=g(x),
\end{array}
\right .
\end{equation}
provided \eqref{eq:HJB} has a solution with some minimal regularity,
according to the usual verification theorems in stochastic optimal control.
When $g$ and $h$ are quadratic functions, this retrieves a linear
quadratic Gaussian control problem for which an explicit solution is
available, see \cite{Bensoussanetal84}.
Otherwise no explicit solution exists and so we have to rely
on numerical methods for non-linear PDEs. 

Consider the specific case where $\bar p = \bar d$, $c^i=\frac{1}{2}$ 
and $\sigma_i=\nu>0$ for any $i=1,\cdots,d$ and $h=0$. By a simple transformation
involving a change of time ($u(t,x):= \frac{1}{2} v(T-t,x)$), we remark that equation~\eqref{eq:HJB} reduces to the KPZ equation 
%%%% TU CITAIS BORKAR .... A QUEL PROPOS?
% for which we have a quasi explicit expression in any dimension $d$.  
%%Borkar V. (2005): Controlled dffusion processes, Probability surveys, 2, pp 213-244
% This PDE makes sense in dimension $d \ge 1$.
% Through this second example, we want to give some empirical evidences that the convergence of
%  $\bar u^{\varepsilon,N,n}$ to $u$, when $N,n \rightarrow + \infty$ and $\varepsilon \rightarrow 0$, 
% remains valid even when our assumptions are not fulfilled. 
% %ensuring well-posedness of our approach can be relaxed.
% Let us consider the KPZ equation
\begin{equation}
\label{eq:KPZ}
\left \{  
\begin{array}{l}
\partial_t u = \frac{\nu^2}{2} \Delta u + \vert \nabla u \vert^2, \quad \textrm{for any}\ (t,x) \in [0,T] \times \R^d, \\
u(0,dx) = u_0(x)dx \ ,
\end{array}
\right .
\end{equation}
where $\Delta$ denotes as usual the Laplace operator  and we recall that $\vert \cdot \vert$ denotes the Euclidean norm on $\R^d$.

Using again the Cole-Hopf transformation, \cite{DelarueMenozzi}
have shown that
 there is a solution $u$ admitting the semi-explicit formula
\begin{eqnarray}
\label{eq:ExpliKPZ}
u(t,x) = \log \Big(\E\big[e^{u_0(x + \sigma B_t)} \big] \Big) \ ,
\end{eqnarray}
where  $B$ denotes   a $\R^d$-valued standard Brownian motion. 
In our numerical tests, \eqref{eq:KPZ} constitutes a benchmark for the stochastic control problem \eqref{eq:inventoryDyn}-\eqref{eq:Bellman}.

%Since it is clear that \eqref{eq:KPZ} is not of the same form than our prototype PDE \eqref{eq:PDE}, we prefer to consider this one
%\begin{equation}
%\label{eq:KPZ2}
%\left \{  
%\begin{array}{l}
%\partial_t u = \frac{\sigma^2}{2} \Delta u + u \Lambda(t,x,u,\nabla_x u), \quad (t,x) \in [0,T] \times \R^d, \\
%u(0,\cdot) = u_0 \ ,
%\end{array}
%\right .
%\end{equation}
We suppose here that the initial condition $u_0$ is chosen strictly 
positive which ensures
% as to ensure 
$u(t,x) \neq 0$ for all $(t,x) \in [0,T] \times \R^d$.
% Observe that this is achieved as soon as $u_0(t,x_0)>0, \; dx_0$-a.e. for all $t \in [0,T]$, which
 Indeed we have
  $e^{u(t,x)}=\displaystyle{\E\big[e^{u_0(x + \sigma B_t)} \big] \geq 1 + \E[u_0(x + \sigma B_t)] > 1}$ for all $(t,x) \in [0,T] \times \R^d$.
We remark that a strictly positive function $u$ is solution of \eqref{eq:KPZ} if and only if it is a solution of equation
%is equivalent to~\eqref{eq:KPZ}, as soon as the  initial condition $u_0$ is strictly positive 
\begin{equation}
\label{eq:KPZ2}
\left \{  
\begin{array}{l}
\partial_t u = \frac{\nu^2}{2} \Delta u + u \Lambda(t,x,u,\nabla u), \quad (t,x) \in [0,T] \times \R^d, \\
\Lambda(t,x,y,z):= \frac{\vert z\vert ^2}{y}\ ,\quad \textrm{for any} \ (t,x,y,z)\in [0,T]\times \R^d\times ]0,+ \infty[\times \R^d\ ,\\
u(0,\cdot) = u_0 \ .
\end{array}
\right .
\end{equation}
%one can  define  $\Lambda(t,x,u,\nabla_x u) := u \vert \frac{\nabla_x_x u}{u} \vert^2$. \eqref{eq:KPZ} is equivalent to \eqref{eq:KPZ2} since 
%In any case, our numerical simulations will be led by choosing an initial condition $u_0$ such that $u(t,x) \neq 0$ 
%for all $(t,x) \in [0,T] \times \R^d$.  
Notice that $\Lambda$ here is clearly not Lipschitz and then it does not satisfy Assumption \ref{ass:main2P3}. 
However,  in our numerical tests,  we have implemented 
% we assume that the 
the time discretized particle scheme~\eqref{eq:tildeYuP3} with the 
 choice of parameters
$\Phi(t,x) := \nu ,\, 
g(t,x) := 0$ and $\Lambda(t,x,y,z):= \frac{\vert z\vert ^2}{y}$, to approximate the solution of~\eqref{eq:KPZ2}.

\subsection{Details of the implementation}
%Although the implementation of the algorithm follows from the steps described in Section \ref{SAlgo}, some details of implementation have to be precised, especially concerning the estimation of the expected $L^1$-norm 
In our figures, we have reported an approximation of the $L^1$-mean error committed by our numerical scheme~\eqref{eq:tildeYuP3} at the terminal time $T$. This error is approximated by Monte Carlo simulations as 
%\begin{eqnarray}
%\label{eq:L1MISE}
%\E \Big[ \Vert u_t^{\varepsilon,N} - u_t \Vert_p \Big] \ ,
%\end{eqnarray}
%for which an explicit formula is not available. \eqref{eq:L1MISE} is then approximated by the Monte Carlo approximation 
\begin{equation}
\label{eq:MISEApproxP3}
\E[ \Vert \bar u_T^{\varepsilon,N,n}-u_T\Vert_1 ] \approx \frac{1}{MQ}\sum_{i=1}^M \sum_{j=1}^Q \vert \bar u_T^{\varepsilon,N,n,i}(X^j)- \hat{u}_T(X^j)\vert \; u^{-1}_0(X^j)\ ,\quad\quad \textrm{where,}
\end{equation}
\begin{itemize}
	\item $(\bar u^{\varepsilon,N,n,i}_T)_{i=1,\cdots ,M=100}$ are i.i.d. estimates based on $M$ i.i.d. particle systems;
	\item $(X^j)_{j=1,\cdots, Q=1000}$ are i.i.d $\R^d$-valued random variables (independent of the particles defining $(\bar u^{\varepsilon,N,n,i}_T)_{i=1,\cdots ,M=100}$), with common density $u_0$;
	\item $\hat u_T$ denotes a Monte Carlo estimation of the exact solution, $u_T$, with $10000$ simulations approximating the expectation formulas~\eqref{eq:ExpliBurg} for the Burgers equation and~\eqref{eq:ExpliKPZ} for the KPZ equation. 
\end{itemize}
The parameters of the problem in both cases (Burgers and KPZ) are  $T=0.1, \nu=0.1$ and the initial distribution $u_0$ is the centered and standard Gaussian distribution $\mathcal{N}(0,I_d)$. \\
Concerning the parameters of our numerical scheme, $n=10$ time steps and $K=\phi^d$ with $\phi^d$ being the standard and centered Gaussian density on $\R^d$. 
To illustrate the trade-off condition (see \eqref{eq:TradeOff}) between $N$ and $\varepsilon$, several values have been considered for the number of particles $N=1000,\,3162,\,10000,\,31623,\,50000$ and for the regularization parameter $\varepsilon=0.1,\, 0.2,\,0.3,\,0.4,\,0.5 ,\, 0.6$.
%We have run a time-discretized version of the particle system with Euler scheme mesh  $kT/n$ with $n=10$. Notice that this discretization error is neglected in the present analysis. \\

\subsection{Simulations results}

We have reported the estimated $L^1$ error (according to~\eqref{eq:MISEApproxP3}) committed by our approximation scheme~\eqref{eq:tildeYuP3} on Figure~\ref{fig:Burgers}, for the Burgers equation~\eqref{eq:Burgers} and on Figure~\ref{fig:KPZ}, for the KPZ equation~\eqref{eq:KPZ}. 
The objective consists in illustrating the tradeoff stated in~\eqref{eq:TradeCond} %deriving from our theoretical bounds of the approximation error 
and to evaluate the convergence rate of the error. 
%%%% CA DESCEND TRES VITE A L'ErREUR DETERMINEE PAR EPSILON
In both cases, one can observe on the left graphs that the error decreases with the number of particles, at a rate $N^{-1/2}$.
 However, when the regularization parameter $\varepsilon$ is big, the largest part of the error is due to $\varepsilon$
 so that the impact of increasing  $N$ is rapidly negligible.

 On  the right-hand side graphs, for fixed $N$, we observe that the error diverges when $\varepsilon$ goes
to zero. As already postulated in Remark \ref{rem:CvgPartic}, the convergence of the error to zero
when $\varepsilon $ goes to zero, holds only letting $N$ goes to infinity according to some relation 
 $N\mapsto \varepsilon(N)$.
The graphs provide empirically the optimal rate $N\mapsto \varepsilon_{opt} (N)$, which 
 corresponds to the value of $\varepsilon$ related to the minimum of the curve indexed by $N$.
%one can observe that the decrease  of $\varepsilon$ to zero should be carefully adjusted to the increase
% of $N$ at an optimal rate, $N\mapsto \varepsilon_{opt} (N)$, which empirically corresponds to the evolution of the minimum of %each curve with $N$.   

We have reported on Figure~\ref{fig:hopt} estimations of these optimal points $(N,\varepsilon_{opt} (N))$
in a  logarithmic scale,
for $N=1000,3162,10000, 31623, 50000$ and drawn a linear interpolation on those points.
The related slopes  are  $-0.21$ (resp. $-0.12$)  for the one dimensional Burgers (resp. the five dimensional KPZ) example. 
 These optimal bandwidths seem to behave accordingly to
 classical kernel density estimation rules, which are of the type $\varepsilon_{opt}\propto \frac{1}{N^{1/(d+4)}}$.
 Indeed $-0.21\approx -1/(d+4)= -1/5$ for the one dimensional
 Burgers example and $-0.12\approx -1/(d+4)=-1/9$ for the five dimensional KPZ example.
 This suggests as already announced in Remark~\ref{rem:CvgPartic} that the tradeoff condition~\eqref{eq:TradeOff} is
 far too rough and that the algorithm behaves better in practice.   
\begin{figure}[!h]
\begin{center}
\subfigure[$L^1$ error as a function of $N$]
{\includegraphics[width=0.49\linewidth]{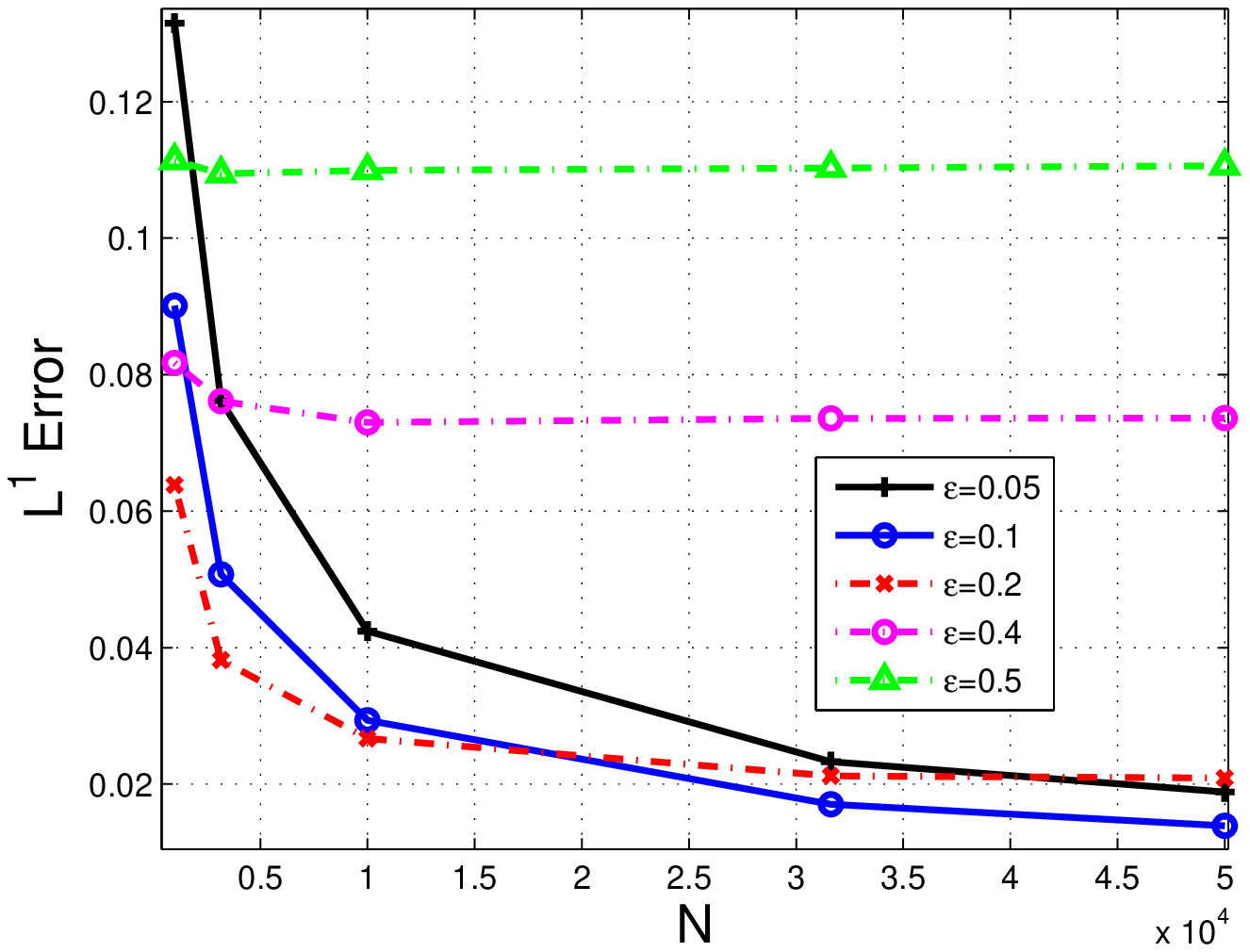}}
\subfigure[$L^1$ error  as a function of $\epsilon$]{\includegraphics[width=0.49\linewidth]{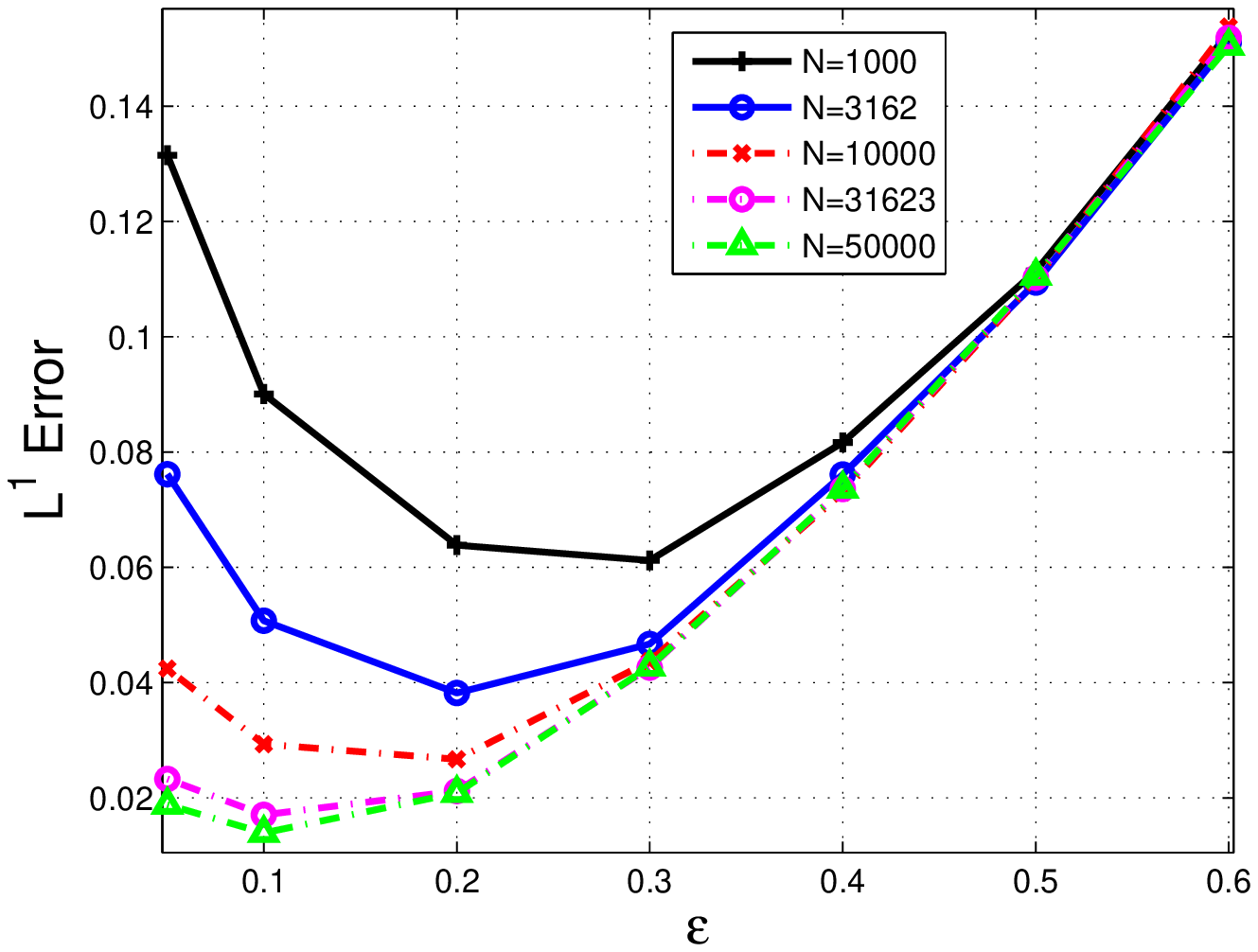}}
\end{center}
\caption{{\small  $L^1$ error as a function of the number of particles, $N$, (on the left graph) and the mollifier window width, $\epsilon$, (on the right graph), for the Burgers equation~\eqref{eq:Burgers}, dimension $d=1$. } }
\label{fig:Burgers}
\end{figure}

\begin{figure}[!h]
\begin{center}
\subfigure[$L^1$ error as a function of $N$]
{\includegraphics[width=0.49\linewidth]{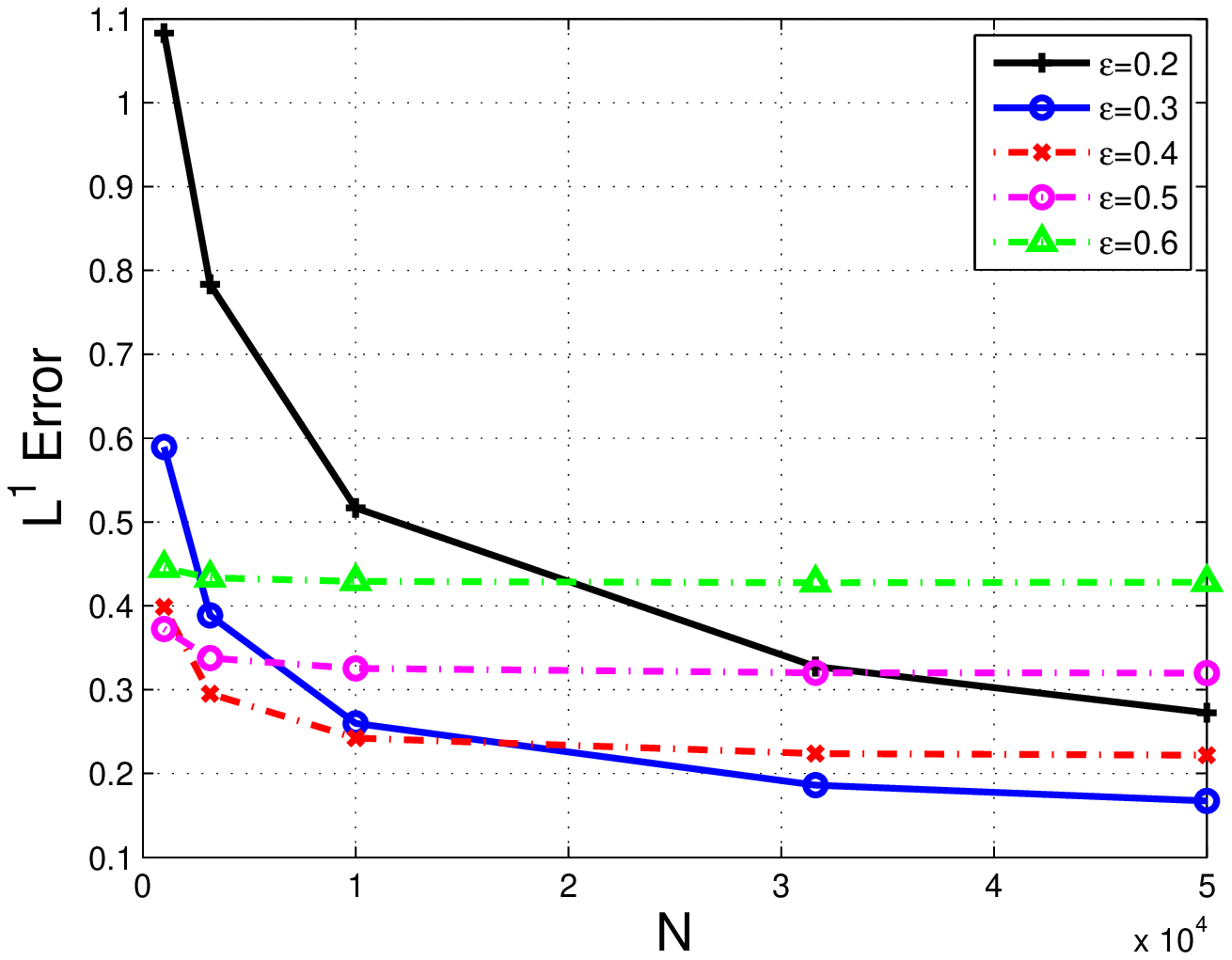}}
\subfigure[$L^1$ error as a function of $\epsilon$]{\includegraphics[width=0.49\linewidth]{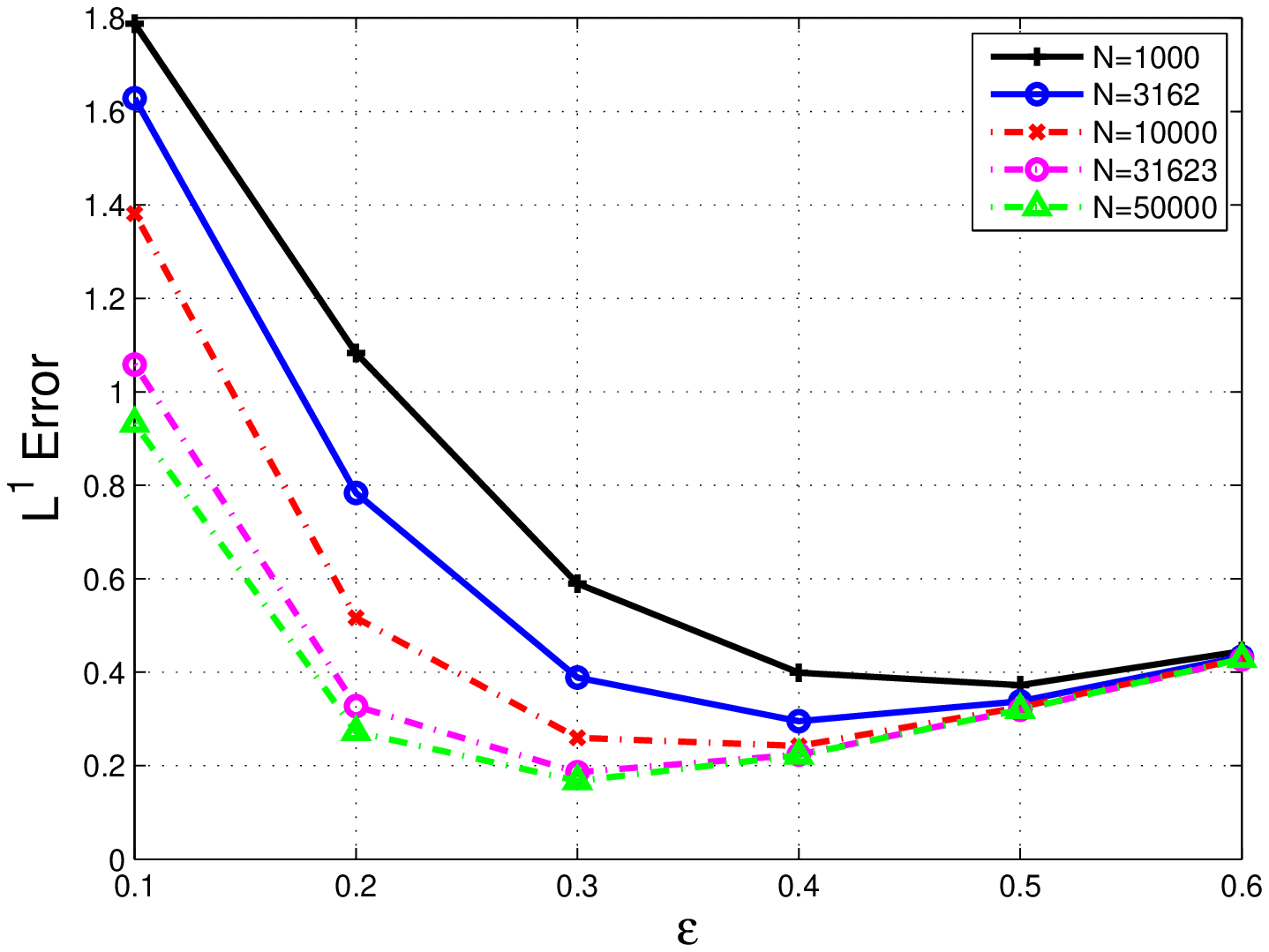}}
\end{center}
\caption{{\small  $L^1$ error as a function of the number of particles, $N$, (on the left graph) and the mollifier window width, $\epsilon$, (on the right graph), for the KPZ equation~\eqref{eq:KPZ}, dimension $d=5$. } }
\label{fig:KPZ}
\end{figure}

\begin{figure}[!h]
\begin{center}
\subfigure[$\varepsilon_{opt}$ as a function of $N$ for Burgers ($d=1$)]
{\includegraphics[width=0.49\linewidth]{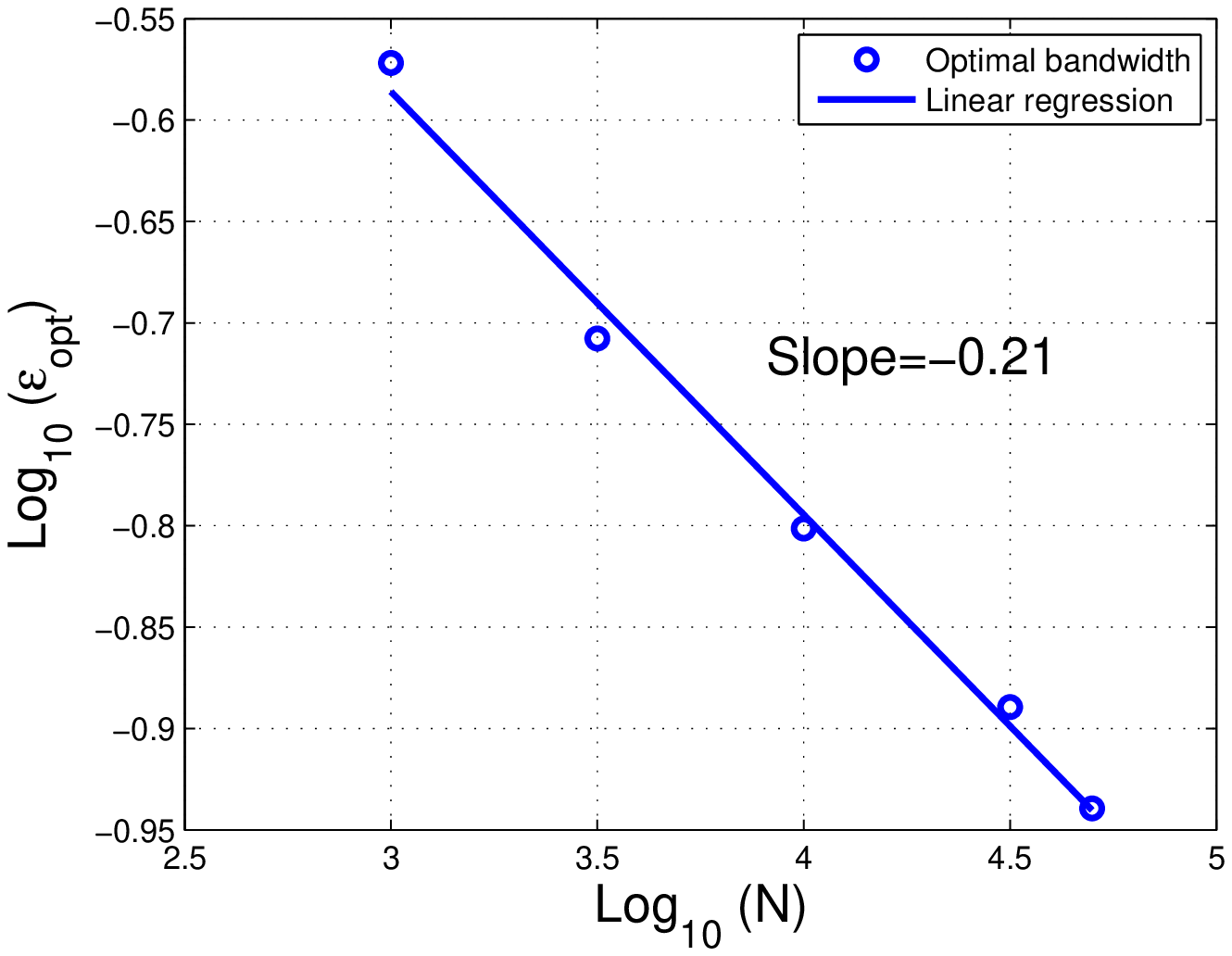}}
\subfigure[$\varepsilon_{opt}$ as a function of $N$ for KPZ ($d=5$)]{\includegraphics[width=0.49\linewidth]{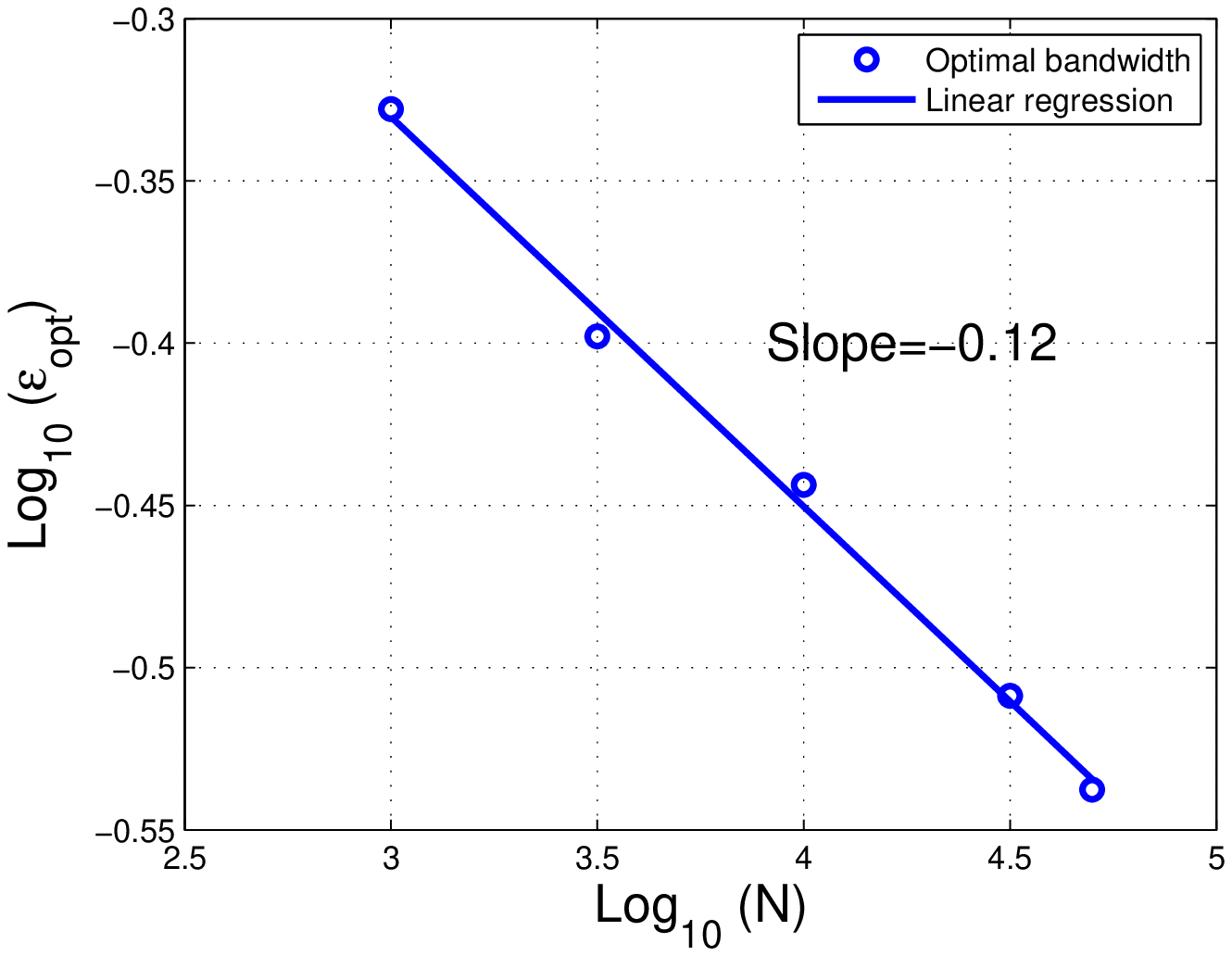}}
\end{center}
\caption{{\small  Optimal bandwidth, $\varepsilon_{opt}$, as a function of the number of particles, for Burgers equation with $d=1$ (left graph) and for the KPZ equation~\eqref{eq:KPZ} with $d=5$ (right graph). } }
\label{fig:hopt}
\end{figure}

\vfill \eject
\section{Appendix}
\label{SAppendixP3}
\setcounter{equation}{0}

\begin{proof}[Proof of Lemma \ref{lem:TimeCvgineq}]
Let us fix $\varepsilon > 0, N \in \N^{\star}$, $t \in [0,T]$. We first recall that for almost all $x \in \R^d$,
\begin{equation}
\label{eq:Defsu}
%\left \{
%\begin{array}{l}
%u^{\varepsilon,N}_t(x) = {\displaystyle \frac{1}{N}\sum_{i=1}^N  K_{\varepsilon}(y-\xi^{i}_t) V_t\big (\xi^{i},u^{\varepsilon,N}(\xi^{i}),\nabla u^{\varepsilon,N}(\xi^{i}) \big ) } \ , \\
\bar{u}^{\varepsilon,N}_{t}(x)=\frac{1}{N}\sum_{i=1}^N K_{\varepsilon}(x-\bar \xi^{i}_{t}) \bar{V}_t\big (\bar{\xi}^{i},\bar{u}^{\varepsilon,N}(\bar{\xi}^{i}),\nabla \bar u^{\varepsilon,N}(\bar{\xi}^{i}) \big ) ,
%\end{array}
%\right .
\end{equation}
%for which $V_t$ (resp. $\bar V_t$) is given by \eqref{eq:VP3} (resp. \eqref{eq:defVbar}). \\
for which  $\bar V_t$ is given by  \eqref{eq:defVbar}. 
Let us fix $i \in \{1,\cdots,N\}$. 
\begin{itemize}
\item {\itshape Proof of \eqref{eq:Lem67Eq6}.}
%The proof of the two inequalities \eqref{eq:Lem67Eq6} is almost the same.
We only give details for the proof of the first inequality
since the second one can be established through similar arguments. \\
From the second line equation of \eqref{eq:Defsu}, we have
\begin{eqnarray}
\vert \bar u^{\varepsilon,N}_{r(t)}(x) - \bar u^{\varepsilon,N}_{r(t)}(y) \vert & \leq & \frac{1}{N} \sum_{i=1}^N \big \vert K_{\varepsilon}(x-\bar \xi^i_{r(t)}) - K_{\varepsilon}(y-\bar \xi^i_{r(t)}) \big \vert \bar{V}_{r(t)}\big (\bar{\xi}^{i},\bar{u}^{\varepsilon,N}(\bar{\xi}^{i}),\nabla \bar u^{\varepsilon,N}(\bar{\xi}^{i}) \big) \nonumber \\
& \leq & \frac{e^{M_{\Lambda}T}}{N\varepsilon^{d+1}} \sum_{i=1}^N L_K \vert x - y \vert \nonumber \\
& \leq & \frac{e^{M_{\Lambda}T}L_K}{\varepsilon^{d+1}} \vert x - y \vert \ ,
\end{eqnarray}
where for the second step above, we have used the fact that $K$ is 
in particular Lipschitz. The same arguments lead also to
\begin{eqnarray}
\vert \nabla \bar u^{\varepsilon,N}_{r(t)}(x) - \nabla \bar u^{\varepsilon,N}_{r(t)}(y) \vert & \leq & \frac{e^{M_{\Lambda}T}L_{\nabla K}}{\varepsilon^{d+2}} \vert x - y \vert \ ,
\end{eqnarray}
which ends the proof of \eqref{eq:Lem67Eq6}.
\item {\itshape Proof of \eqref{eq:Lem67Eq8}.}
From
\begin{equation}
\label{eq:defu}
\bar u^{\varepsilon,N}_{t}(x) = {\displaystyle \frac{1}{N}\sum_{i=1}^N  K_{\varepsilon}(x-\bar \xi^{i}_t) \bar V_t\big (\bar \xi^{i},\bar u^{\varepsilon,N}(\bar \xi^{i}),\nabla \bar u^{\varepsilon,N}(\bar \xi^{i}) \big ) } \ , \quad x \in \R^d \ ,
\end{equation}
% and
% \begin{equation}
% \label{eq:defu2}
% \bar u^{\varepsilon,N}_{r(t)}(x) = {\displaystyle \frac{1}{N}\sum_{i=1}^N  K_{\varepsilon}(x-\bar \xi^{i}_{r(t)}) \bar V_{r(t)}\big (\bar \xi^{i},u^{\varepsilon,N}(\bar \xi^{i}),\nabla u^{\varepsilon,N}(\bar \xi^{i}) \big ) } \ , \quad x \in \R^d \ ,
% \end{equation}
we deduce, for almost all $x \in \R^d$,
\begin{eqnarray}
\vert \bar u^{\varepsilon,N}_t(x) - \bar u^{\varepsilon,N}_{r(t)}(x) \vert & \leq & \frac{e^{M_{\Lambda}T}}{N} \sum_{i=1}^N \Big \vert K_{\varepsilon}(x-\bar \xi^{i}_t) - K_{\varepsilon}(x- \bar \xi^{i}_{r(t)}) \Big \vert  \nonumber \\
&& \; + \; \frac{\Vert K \Vert_{\infty}}{N \varepsilon^d} \sum_{i=1}^N \big \vert \bar V_t\big ( \bar \xi^{i},\bar u^{\varepsilon,N}(\bar \xi^{i}),\nabla \bar u^{\varepsilon,N}(\bar \xi^{i}) \big ) - \bar V_{r(t)}\big (\bar \xi^{i},\bar u^{\varepsilon,N}(\bar \xi^{i}),\nabla \bar u^{\varepsilon,N}(\bar \xi^{i}) \big ) \big \vert \ . \nonumber \\
\end{eqnarray}
Since $K$ is Lipschitz with  related constant  $L_K = \Vert \nabla K \Vert_\infty$,
for almost all $x \in \R^d$, we obtain
\begin{eqnarray}
\label{eq:638}
\vert \bar u^{\varepsilon,N}_t(x) - \bar u^{\varepsilon,N}_{r(t)}(x) \vert & \leq & \frac{L_K e^{M_{\Lambda}T}}{N\varepsilon^{d+1}} \sum_{i=1}^N  \vert \bar \xi^i_t - \bar \xi^i_{r(t)} \vert \nonumber \\
&& \; + \; \frac{L_{\Lambda}e^{M_{\Lambda}T} \Vert K \Vert_{\infty}}{N \varepsilon^d} \sum_{i=1}^N \int_{r(t)}^t \Lambda(r(s),\bar \xi^i_{r(s)}, \bar u^{\varepsilon,N}_{r(s)}(\bar \xi^i_{r(s)}), \nabla \bar u^{\varepsilon,N}_{r(s)}(\bar \xi^i_{r(s)})) ds \ , \nonumber \\
\end{eqnarray}
where the second term in \eqref{eq:638} comes from inequality \eqref{eq:LipV}. Since $\Lambda$ is bounded, by taking the supremum w.r.t. $x$ and the expectation in both sides of inequality above we have
\begin{eqnarray}
\E \Big[ \Vert \bar u^{\varepsilon,N}_t - \bar u^{\varepsilon,N}_{r(t)} \Vert_{\infty} \Big] & \leq & \frac{L_K e^{M_{\Lambda}T}}{N\varepsilon^{d+1}} \sum_{i=1}^N \E \Big[ \vert \bar \xi^i_t - \bar \xi^i_{r(t)} \vert \Big] \nonumber \\
&& \; + \; \frac{L_{\Lambda}e^{M_{\Lambda}T} \Vert K \Vert_{\infty}}{\varepsilon^d} M_{\Lambda} \delta t 
 \leq  \frac{C \sqrt{\delta t}}{\varepsilon^{d+1}} \ ,
\end{eqnarray}
where we have used the fact that $\E \Big[ \vert \bar \xi^i_s - \bar \xi^i_{r(s)} \vert^2 \Big] \leq C \delta t$, since $\Phi, g$ are bounded.
% see \eqref{E537} in Proposition \ref{prop:CvgTime}. 
\\
The bound of $ \E \Big[ \Vert \nabla \bar u^{\varepsilon,N}_t - \nabla \bar u^{\varepsilon,N}_{r(t)} \Vert_{\infty} \Big] $ is obtained by proceeding exactly in with the same way as above, starting with 
\begin{equation}
\frac{\partial \bar u^{\varepsilon,N}_t}{\partial x_{\ell}}(\cdot) = {\displaystyle \frac{1}{N \varepsilon}\sum_{i=1}^N  \frac{\partial K_{\varepsilon}}{\partial x_{\ell}}( \cdot-\bar \xi^{i}_t) \bar V_t\big (\bar \xi^{i}, \bar u^{\varepsilon,N}(\bar \xi^{i}),\nabla \bar u^{\varepsilon,N}(\bar \xi^{i}) \big ) } \ , l = 1, \cdots, d \ ,
\end{equation}
instead of \eqref{eq:defu}, where $x_\ell$ denotes the 
$\ell$-th coordinate of $x \in \R^d$. It follows then
\begin{eqnarray}
\E \Big[ \Vert \nabla \bar u^{\varepsilon,N}_t - \bar \nabla u^{\varepsilon,N}_{r(t)} \Vert_{\infty} \Big] & \leq & \frac{C \sqrt{\delta t}}{\varepsilon^{d+2}} \ .
\end{eqnarray}
\end{itemize}
\end{proof}

% {\bf ACKNOWLEDGMENTS.} Part of this work has been done during a stay of the third named author in the University of Bielefeld, %SFB 701. He is grateful to Prof. Michael R\"ockner for
% the kind invitation.

\addcontentsline{toc}{section}{Bibliography}
\bibliographystyle{plain}
%\bibliography{../../NonConservativePDE_bib/NonConservativePDE}
\bibliography{NonConservativePDE}
\end{document}